\documentclass[12pt]{amsart}
\newcommand{\deleted}[1]{}
\newcommand{\delete}[1]{}
\newcommand{\mynotes}[1]{}
\newcommand\notes[1]{}
\usepackage{txfonts}
\usepackage{amssymb}
\usepackage{eucal}
\usepackage{graphicx}
\usepackage{amsmath}
\usepackage{amscd}
\usepackage[all]{xy}           

\usepackage{amsfonts,latexsym}
\usepackage{xspace}
\usepackage{epsfig}
\usepackage{float}
\usepackage{color}
\usepackage{fancybox}
\usepackage{colordvi}
\usepackage{multicol}
\usepackage{tikz}

\usepackage{wasysym}
\usepackage{xspace}
\usepackage[colorlinks,final,backref=page,hyperindex,hypertex]{hyperref}

\usepackage[active]{srcltx} 

\newcommand\changed[1]{#1}
\changed{}

\newtheorem{theorem}{Theorem}[section]
\newtheorem{lemma}[theorem]{Lemma}
\newtheorem{coro}[theorem]{Corollary}
\newtheorem{problem}[theorem]{Problem}

\newtheorem{prop}[theorem]{Proposition}
\theoremstyle{definition}
\newtheorem{defn}[theorem]{Definition}
\newtheorem{remark}[theorem]{Remark}
\newtheorem{exam}[theorem]{Example}

\deleted{\newtheorem{theorem}{Theorem}[section]
\newtheorem{prop-def}{Proposition-Definition}[section]
\newtheorem{coro-def}{Corollary-Definition}[section]

}

\newcommand{\nc}{\newcommand}
\nc{\tred}[1]{\textcolor{red}{#1}} \nc{\tblue}[1]{\textcolor{blue}{#1}} \nc{\tgreen}[1]{\textcolor{green}{#1}} \nc{\tpurple}[1]{\textcolor{purple}{#1}} \nc{\btred}[1]{\textcolor{red}{\bf #1}} \nc{\btblue}[1]{\textcolor{blue}{\bf #1}} \nc{\btgreen}[1]{\textcolor{green}{\bf #1}} \nc{\btpurple}[1]{\textcolor{purple}{\bf #1}}

\renewcommand{\Bbb}{\mathbb}

\newcommand{\efootnote}[1]{}

\newcommand\wyscco[1]{}
\renewcommand{\textbf}[1]{}

\nc{\mlabel}[1]{\label{#1}}  
\nc{\mcite}[1]{\cite{#1}}  
\nc{\mref}[1]{\ref{#1}}  
\nc{\mbibitem}[1]{\bibitem{#1}} 

\delete{
\nc{\mlabel}[1]{\label{#1}{\hfill \hspace{1cm}{\bf{{\ }\hfill(#1)}}}}
\nc{\mcite}[1]{\cite{#1}{{\bf{{\ }(#1)}}}}  
\nc{\mref}[1]{\ref{#1}{{\bf{{\ }(#1)}}}}  
\nc{\mbibitem}[1]{\bibitem[\bf #1]{#1}} 
}

\renewcommand\geq{\geqslant}
\renewcommand\leq{\leqslant}

\renewcommand\bar[1]{\overline{#1}}
\renewcommand\tilde[1]{\widetilde{#1}}

\nc\kdot{\bfk}
\nc\simple{simple\xspace}

\nc{\rbw}{\mathfrak{R}} \nc{\brp}{\mathrm{brp}} \nc{\lead}{\mathrm{Lead}} \nc{\Id}{\mathrm{Id}} \nc{\Irr}{\mathrm{Irr}} \nc{\vx}{\sigma} \nc{\vy}{\tau} \nc{\dvx}{\sigma^{(1)}} \nc{\dvy}{\tau^{(1)}} \nc{\done}{\vep} \nc{\citep}[1]{\cite{#1}} \nc{\wt}{\mathrm{wt}} \nc{\bre}[1]{|#1|} \nc{\mapmonoid}{\frakM} \nc{\disjoint}{\frakM'}
\nc{\ncpoly}[1]{\langle #1\rangle}  
\nc{\mapm}[1]{\frakM(#1)}
\nc{\diff}[1]{{}^\NC\{ #1 \}} \nc{\disj}[1]{\{{#1}\}'} \nc{\mdisj}[1]{\frakM'(#1)} \nc{\brho}{\bar{\rho}} \nc{\om}{\bar{\frakm}} \nc{\frakn}{\mathfrak n} \nc{\ddeg}[1]{^{(#1)}} \nc{\opset}{X} \nc{\genset}{{Z}} \nc{\NC}{\mathrm{{NC}}} \nc{\leaf}{\mathrm{leaf}} \nc{\twig}{\mathrm{twig}} \nc{\fe}{\mathrm{fl}} \nc{\munderline}[1]{#1} \nc{\bo}{o} \nc{\dep}{\mathrm{dep}} \nc{\ofe}{\mathrm{ofl}} \nc{\dfe}{\mathrm{dfe}} \nc{\fex}{\mathrm{fex}} \nc{\dl}{\mathrm{dlex}} \nc{\db}{\mathrm{db}} \nc{\lex}{\mathrm{lex}} \nc{\clex}{\mathrm{clex}} \nc{\dgp}{\mathrm{dgp}} \nc{\dgx}{\mathrm{dgx}} \nc{\br}{\mathrm{br}} \nc{\obd}{\mathrm{odb}} \nc{\ob}{\mathrm{ob}}
\nc{\loc}{location\xspace}
\nc{\occ}{occurrence\xspace}
\nc{\occs}{occurrences\xspace}
\nc{\pla}{placement\xspace}
\nc{\plas}{placements\xspace}

\nc{\bin}[2]{ (_{\stackrel{\scs{#1}}{\scs{#2}}})}  
\nc{\binc}[2]{ \left (\!\! \begin{array}{c} \scs{#1}\\
    \scs{#2} \end{array}\!\! \right )}  
\nc{\bincc}[2]{  \left ( {\scs{#1} \atop
    \vspace{-1cm}\scs{#2}} \right )}  
\nc{\bs}{\bar{S}} \nc{\cosum}{\sqsubset} \nc{\la}{\longrightarrow} \nc{\rar}{\rightarrow} \nc{\dar}{\downarrow} \nc{\dprod}{**} \nc{\dap}[1]{\downarrow \rlap{$\scriptstyle{#1}$}} \nc{\md}{\mathrm{dth}} \nc{\uap}[1]{\uparrow \rlap{$\scriptstyle{#1}$}} \nc{\defeq}{\stackrel{\rm def}{=}} \nc{\disp}[1]{\displaystyle{#1}} \nc{\dotcup}{\ \displaystyle{\bigcup^\bullet}\ } \nc{\gzeta}{\bar{\zeta}} \nc{\hcm}{\ \hat{,}\ } \nc{\hts}{\hat{\otimes}} \nc{\barot}{{\otimes}} \nc{\free}[1]{\widetilde{#1}} \nc{\uni}[1]{\tilde{#1}} \nc{\hcirc}{\hat{\circ}} \nc{\leng}{\ell} \nc{\lleft}{[} \nc{\lright}{]} \nc{\lc}{\lfloor} \nc{\rc}{\rfloor}
\nc{\lb}{[} 
\nc{\rb}{]} 
\nc{\curlyl}{\left \{ \begin{array}{c} {} \\ {} \end{array}
    \right.  \!\!\!\!\!\!\!}
\nc{\curlyr}{ \!\!\!\!\!\!\!
    \left. \begin{array}{c} {} \\ {} \end{array}
    \right \} }
\nc{\longmid}{\left | \begin{array}{c} {} \\ {} \end{array}
    \right. \!\!\!\!\!\!\!}
\nc{\onetree}{\bullet} \nc{\ora}[1]{\stackrel{#1}{\rar}}
\nc{\ola}[1]{\stackrel{#1}{\la}}
\nc{\ot}{\otimes} \nc{\mot}{{{\boxtimes\,}}} \nc{\otm}{\overline{\boxtimes}} \nc{\sprod}{\bullet} \nc{\scs}[1]{\scriptstyle{#1}} \nc{\mrm}[1]{{\rm #1}} \nc{\msum}{\sum\limits}
\nc{\margin}[1]{\marginpar{\rm #1}}   
\nc{\dirlim}{\displaystyle{\lim_{\longrightarrow}}\,} \nc{\invlim}{\displaystyle{\lim_{\longleftarrow}}\,} \nc{\mvp}{\vspace{0.3cm}} \nc{\tk}{^{(k)}} \nc{\tp}{^\prime} \nc{\ttp}{^{\prime\prime}} \nc{\svp}{\vspace{2cm}} \nc{\vp}{\vspace{8cm}} \nc{\proofbegin}{\noindent{\bf Proof: }}
\nc{\proofend}{$\blacksquare$ \vspace{0.3cm}}
\nc{\modg}[1]{\!<\!\!{#1}\!\!>}
\nc{\intg}[1]{F_C(#1)} \nc{\lmodg}{\!<\!\!} \nc{\rmodg}{\!\!>\!} \nc{\cpi}{\widehat{\Pi}}
\nc{\sha}{{\mbox{\cyr X}}}  
\nc{\shap}{{\mbox{\cyrs X}}} 
\nc{\shpr}{\diamond}    
\nc{\shp}{\ast} \nc{\shplus}{\shpr^+}
\nc{\shprc}{\shpr_c}    
\nc{\msh}{\ast} \nc{\zprod}{m_0} \nc{\oprod}{m_1} \nc{\vep}{\varepsilon} \nc{\labs}{\mid\!} \nc{\rabs}{\!\mid}
\nc{\astarrow}{\overset{\raisebox{-2pt}{{\scriptsize $\ast$}}}{\rightarrow}}
\nc{\astlarrow}{\overset{\raisebox{-2pt}{{\scriptsize $\ast$}}}{\longrightarrow}}
\nc{\lastarrow}{\overset{\raisebox{-2pt}{{\scriptsize $\ast$}}}{\leftarrow}}
\nc{\mastarrow}[1]{\overset{\raisebox{-2pt}{{\scriptsize $#1$}}}{\rightarrow}}
\nc{\quvarrow}[3]{#1 \overset{q,u,v}{\longrightarrow}_{#3} #2}
\nc{\quvkto}[1]{f_{#1} \overset{q_{#1}, u_{#1}, v_{#1}}{\longrightarrow}_\phi g_{#1}}
\nc{\tvarrow}[3]{#1 \overset{(t,v)}{\longrightarrow}_{#3} #2}
\nc{\Supp}{{\rm Supp}}
\nc{\mpu}{u^{\ast}}
\nc{\mpv}{v^{\ast}}
\nc{\mpw}{w^{\ast}}
\nc{\mpx}{x^{\ast}}
\nc{\dps}{\dotplus}

\nc{\dth}{d} \nc{\mmbox}[1]{\mbox{\ #1\ }} \nc{\fp}{\mrm{FP}} \nc{\rchar}{\mrm{char}} \nc{\Fil}{\mrm{Fil}} \nc{\Mor}{Mor\xspace} \nc{\gmzvs}{gMZV\xspace} \nc{\gmzv}{gMZV\xspace} \nc{\mzv}{MZV\xspace} \nc{\mzvs}{MZVs\xspace} \nc{\Hom}{\mrm{Hom}} \nc{\id}{\mrm{id}} \nc{\im}{\mrm{im}} \nc{\incl}{\mrm{incl}} \nc{\map}{\mrm{Map}} \nc{\mchar}{\rm char} \nc{\nz}{\rm NZ} \nc{\supp}{\mathrm Supp}
\nc{\mo}{\mathbf o}
\nc{\pl}{\mathfrak{p}}

\nc{\Alg}{\mathbf{Alg}} \nc{\Bax}{\mathbf{Bax}} \nc{\bff}{\mathbf f} \nc{\bfk}{{\bf k}} \nc{\bfone}{{\bf 1}} \nc{\bfx}{\mathbf x} \nc{\bfy}{\mathbf y}
\nc{\base}[1]{\bfone^{\otimes ({#1}+1)}} 
\nc{\Cat}{\mathbf{Cat}} \delete{}
\nc{\detail}{\marginpar{\bf More detail}
    \noindent{\bf Need more detail!}
    \svp}
\nc{\Int}{\mathbf{Int}} \nc{\Mon}{\mathbf{Mon}}
\nc{\rbtm}{{shuffle }} \nc{\rbto}{{Rota-Baxter }} \nc{\remarks}{\noindent{\bf Remarks: }} \nc{\Rings}{\mathbf{Rings}} \nc{\Sets}{\mathbf{Sets}}
\nc{\vwpt}{{Let $V$ be a free $\bfk$-module with a $\bfk$-basis $W$ and let $\Pi$ be a \simple term-rewriting system on $V$ with respect to $W$.}\xspace}

\nc{\BA}{{\Bbb A}} \nc{\CC}{{\Bbb C}} \nc{\DD}{{\Bbb D}} \nc{\EE}{{\Bbb E}} \nc{\FF}{{\Bbb F}} \nc{\GG}{{\Bbb G}} \nc{\HH}{{\Bbb H}} \nc{\LL}{{\Bbb L}} \nc{\NN}{{\Bbb N}} \nc{\KK}{{\Bbb K}} \nc{\QQ}{{\Bbb Q}} \nc{\RR}{{\Bbb R}} \nc{\TT}{{\Bbb T}} \nc{\VV}{{\Bbb V}} \nc{\ZZ}{{\Bbb Z}}

\nc{\cala}{{\mathcal A}} \nc{\calc}{{\mathcal C}} \nc{\cald}{{\mathcal D}} \nc{\cale}{{\mathcal E}} \nc{\calf}{{\mathcal F}} \nc{\calg}{{\mathcal G}} \nc{\calh}{{\mathcal H}} \nc{\cali}{{\mathcal I}} \nc{\call}{{\mathcal L}} \nc{\calm}{{\mathcal M}} \nc{\caln}{{\mathcal N}} \nc{\calo}{{\mathcal O}} \nc{\calp}{{\mathcal P}} \nc{\calr}{{\mathcal R}} \nc{\cals}{{\mathcal S}} \nc{\calt}{{\mathcal T}} \nc{\calw}{{\mathcal W}}
\nc{\calv}{{\mathcal V}}
\nc{\calk}{{\mathcal K}} \nc{\calx}{{\mathcal X}} \nc{\CA}{\mathcal{A}}

\nc{\fraka}{{\mathfrak a}} \nc{\frakA}{{\mathfrak A}} \nc{\frakb}{{\mathfrak b}} \nc{\frakB}{{\mathfrak B}} \nc{\frakC}{{\mathfrak C}}
\nc{\frakD}{{\mathfrak D}} \nc{\frakH}{{\mathfrak H}} \nc{\frakM}{{\mathfrak M}} \nc{\bfrakM}{\overline{\frakM}} \nc{\frakm}{{\mathfrak m}} \nc{\frkP}{{\mathfrak P}}
\nc{\frakN}{{\mathfrak N}} \nc{\frakp}{{\mathfrak p}} \nc{\fraku}{{\mathfrak u}} \nc{\frakv}{{\mathfrak v}}
\nc{\frakQ}{{\mathfrak Q}}\nc{\frakR}{{\mathfrak R}} \nc{\frakS}{{\mathfrak S}}
\nc{\frakx}{{\mathfrak x}} \nc{\ox}{\bar{\frakx}} \nc{\frakX}{{\mathfrak X}} \nc{\fraky}{{\mathfrak y}}
\nc\dop{\delta}
\nc{\Reduce}{{\rm Red}}

\font\cyr=wncyr10 \font\cyrs=wncyr7
\nc{\redt}[1]{\textcolor{red}{#1}}

\nc{\li}[1]{\textcolor{red}{Li:#1}} 
\nc{\lio}[1]{}
\nc{\sz}[1]{\textcolor{green}{sz:#1}}
\nc{\szo}[1]{}
\nc{\xing}[1]{\textcolor{purple}{Xing:#1}}
\nc{\ws}[1]{\textcolor{blue}{{#1}}} 
\nc{\wsc}[1]{\textcolor{blue}{ws:#1}} 
\nc{\wsco}[1]{}
\nc{\wsn}[1]{\textcolor{magenta}{#1}} 

\nc{\medmid}{{\,~{\tiny \longmid}~\,}}
 \nc{\lbar}[1]{\overline{#1}}
\nc{\anf}{\bf $\phi$-NF} \nc{\lto}{\longrightarrow}
\nc{\gs}{Gr\"{o}bner-Shirshov\xspace}
\nc{\Gs}{Gr\"{o}bner-Shirshov\xspace}
\nc{\pgs}{potentially Gr\"{o}bner-Shirshov\xspace}
\nc{\Pgs}{Potentially Gr\"{o}bner-Shirshov\xspace}
\nc{\egs}{essentially Gr\"obner-Shirshov\xspace}
\nc{\cs}{convergent\xspace}
\nc{\pcs}{potentially convergent\xspace}
\nc{\ecs}{essentially convergent\xspace}

\nc{\brw}{\frakM(Z)} \nc{\irr}{{\rm Irr}} \nc{\pis}{\Pi_S}
\nc{\term}{term\xspace} \nc{\Term}{Term\xspace}
\nc{\re}[1]{R(#1)} \nc{\sumre}[2]{R^{#1}_{#2}}
\nc{\stars}[2]{#1|_{#2}}
\nc{\nbfk}{\bfk^{\times}} \nc{\revise}[1]{\textcolor{blue}{#1}}
\nc{\ord}{{\rm ord}} \nc{\tpi}{\to_{\Pi}}
\nc{\fix}[1]{\tilde{#1}}

\begin{document}
\title[Rota's Classification Problem, rewriting systems and Gr\"obner-Shirshov bases]{Rota's Classification Problem, rewriting systems and
Gr\"obner-Shirshov bases}

\author{Xing Gao} \address{School of Mathematics and Statistics,
Key Laboratory of Applied Mathematics and Complex Systems,
Lanzhou University, Lanzhou, 730000, P.R. China}
\email{gaoxing@lzu.edu.cn}

\author{Li Guo}\address{Department of Mathematics and Computer Science,
Rutgers University, Newark, NJ 07102, USA}
\email{liguo@rutgers.edu}

\hyphenpenalty=8000
\date{\today}

\begin{abstract}
In this paper we revisit Rota's Classification Problem on classifying algebraic identities for linear operator. We reformulate Rota's Classification Problem in the contexts of rewriting systems and Gr\"{o}bner-Shirshov bases,
through which Rota's Classification Problem amounts to the classification of operators, given by their defining operator identities, that give convergent rewriting systems or Gr\"obner-Shirshov bases.
Relationship is established between the reformulations in terms of rewriting systems and that of Gr\"{o}bner-Shirshov bases.
We provide an effective condition that gives \gs  operators and obtain a new class of \gs  operators.
\end{abstract}

\subjclass[2010]{16W99, 13P10, 16S15, 12H05,
08A70 
16S20 
16R99 
}

\keywords{
Rota's Classification Problem, linear operators, operator identities, Gr\"ober-Shirshov bases, term rewriting systems, normal forms, free objects.
}

\maketitle
\vspace{-1.2cm}

\tableofcontents

\vspace{-1.3cm}

\allowdisplaybreaks


\section{Introduction}

\subsection{Motivation}
Motivated by the important roles played by various linear operators in the study of mathematics through their actions on objects,  Rota~\cite{Ro2} posed the problem of

\begin{quote}
finding all possible \underline{algebraic identities} that can be satisfied by a linear operator on an \underline{algebra},
\label{qu:rota}
\end{quote}
henceforth called {\bf Rota's Classification Problem}.

Operator identities that were interested to Rota included
\begin{eqnarray}
 \text{Endomorphism operator} &\quad&
d(xy)=d(x)d(y), \nonumber\\
 \text{Differential operator} &\quad&
d(xy)=d(x)y+xd(y), \nonumber\\
 \text{Average operator} &\quad&
P(x)P(y)=P(xP(y)),  \nonumber\\
\text{Inverse average operator} &\quad& P(x)P(y)=P(P(x)y),
\nonumber\\
 \text{(Rota-)Baxter operator of weight\ $\lambda$} &\quad&
 P(x)P(y)=P(xP(y)+P(x)y+\lambda xy),\nonumber\\
 &\quad &\quad \text{ where } \lambda \text{ is a fixed constant},\nonumber\\
\text{Reynolds operator}
  &\quad& P(x)P(y)=P(xP(y)+P(x)y-P(x)P(y)). \nonumber
\end{eqnarray}

After Rota posed his problem, more operators have appeared, such as \begin{eqnarray}
\text{Differential operator of weight} ~ \lambda &\quad&
 d(xy)=d(x)y+x d(y)+\lambda d(x)d(y), \nonumber\\
&\quad& \quad\text{ where } \lambda \text{ is a fixed constant},\nonumber\\
\text{Nijenhuis operator} &\quad& P(x)P(y)=P(xP(y)+P(x)y-P(xy)),\nonumber\\
\text{Leroux's TD operator} &\quad&
P(x)P(y)=P(xP(y)+P(x)y-xP(1)y).\nonumber
\end{eqnarray}

The pivotal roles played by the endomorphisms (such as in Galois theory) and derivations (such as in calculus) are well-known.
Their abstractions have led to the concepts of difference algebra and differential algebra respectively.
The other operators also found applications in a broad range of pure and applied mathematics, including combinatorics, probability and mathematical physics~\cite{BBGN,Ba,CGM,CK1,Gub,Kol,Mill,Nij,PS,Re,Ro2}. See~\cite{GGSZ,GSZ} for further references.

These sustained interests in linear operators that satisfy special operator identities warrant a systematic study of Rota's Classification Problem, leading to the articles~\cite{GGSZ,GSZ}. There are multiple benefits in such study, on the one hand to find a uniform approach to these various existing operators and on other other hand to understand the nature of these operators, namely what distinguish them from a randomly taken operator identity. The latter also sheds light on possible new operator identities that might arise in mathematics and its applications.

\subsection{Rota's Classification Problem in special cases}
There are two stages in the recent approach to Rota's Classification Problem. The first stage is to establish an algebraic framework in which to consider \underline{algebraic identities} satisfied by a linear operator in Rota's Classification Problem. As a prototype, we recall that an algebraic identity satisfied by an algebra is an element in a noncommutative polynomial algebra, as a realization of a free (associative) algebra, leading to the extensive study of polynomial identity (PI) rings~\cite{DF, Pr, Ro}. Since there is an operator involved in an algebraic identity in Rota's Classification Problem, we take an algebraic identity satisfied by an operator to be an element in a free object in the category of algebras with an operator, or operated algebras, whose origin can be tracked back to Kurosh~\cite{Ku}. In~\cite{Gop}, such a free object is realized in the form of polynomials in variables together with their formal derivations, amendable to serve as operated polynomial identities (OPIs) for an algebra with operators.

In this sense, all the operators list above are defined by OPIs. This naturally leads to the second stage in our understanding of Rota's Classification Problem: what distinguishes the OPIs satisfied by these operators from the OPIs defined by arbitrary elements from the operated polynomial algebras? This is a key difference between PI algebras and OPI algebras. In the study of the former, not much difference is made among the elements in the polynomial algebras. This is apparent not the case for elements from the operated polynomial algebras, hence Rota's Classification Problem. In other words, Rota apparently asked to identify special OPIs that are worth of further study, as in the case of the OPIs in the above lists. As a hint for what to look for in these ``good" OPIs, we pay special attention that Rota's Classification Problem asks for linear operators defined on an \underline{algebra}, which in his context means an associative algebra. Therefore, such a ``good" OPIs should satisfy certain compatibility condition with the associativity of the algebra that the operator acts on.

In order to make sense of this compatibility for arbitrary OPIs, we first tested two classes of OPIs which, despite their special forms, are general enough to cover all the operators considered above, except the Reynolds operator. The two classes of operators are called the differential type operators and Rota-Baxter type operators, for their resemblance to the differential operator and the Rota-Baxter operator respectively.

As the initial step, differential type operators, the easier of the two classes of operators, were studied in~\cite{GSZ}, revealing that, the seemingly vague and specialized problem of Rota can be casted in completely general setups. First of all, it was showed that, the somehow ad hoc properties defining differential type operators turn out to be equivalence to the convergence of the rewriting systems defined by these operators. Second, these properties are also equivalent to the existence of a generalization of the Gr\"obner basis, called the Gr\"obner-Shirshov basis, for the ideals defined by these OPIs, giving rise to an explicit construction of the free objects in the category of the algebras satisfying the OPIs. These equivalences suggest intimate connection from Rota's Classification Problem to rewriting systems and Gr\"obner bases.

To obtain more evidence for this speculation, the class of Rota-Baxter type operators was studied in~\cite{GGSZ}. It is encouraging to see that, despite the much more challenging nature of Rota-Baxter operators, the same connections can be established from them to convergent rewriting systems on the one hand and to Gr\"obner-Shirshov bases on the other.

\subsection{Rota's Classification Problem in the general case}
The success in characterizing these two important classes of operators in terms of general properties in rewriting systems and ideal generators motivates us to understand Rota's Classification Problem in the context of these general properties for OPIs, rather than by certain special forms such as being of the differential type or Rota-Baxter type. We carry out this approach in this paper.

We give, in Section~\ref{sec:sgood}, two formulations of Rota's Classification Problem for desirable systems of operator identities, one in terms of convergent rewriting systems and one in terms of Gr\"obner-Shirshov bases. When one monomial in an operated identity is chosen as the leading term, the identity gives a rewriting rule. Our first formulation of Rota's Classification Problem is to find OPIs for which one rewriting system obtained this way is convergent (Problem~\mref{pr:rotare}).

An important and effective way to determine the convergency of a rewriting system is the method of Gr\"obner bases in the case of commutative algebras, or Gr\"obner-Shirshov bases in general. Thus our second formulation of Rota's Classification Problem is to find systems of OPIs that are Gr\"obner-Shirshov bases of the operated ideals that these systems generate, leading to the concepts of \gs and \pgs systems of OPIs, and the corresponding \gs and \pgs operators (Problem~\mref{pr:rota}).

In Section~\mref{sec:sgoodrew}, we establish the relationship between the two reformulations of Rota's Classification Problem, by showing that a \gs system of OPIs gives a convergent system (Theorem~\ref{thm:sgoodtr}). The interplay between the two systems proves to be fruitful. For example, it is not hard to show that the OPIs for the two-sided averaging operator is not
\cs and hence not \gs (Corollary~\ref{coro:avera}); while from showing that it is \pgs we conclude that it is \pcs (Remark~\ref{re:apconv}).

This conceptual approach allows us to obtain an effective criterion to obtain \gs operators (Theorem~\mref{thm:sgoodtype}), including not only the two previously known differential type and Rota-Baxter type operators, but also the modified Rota-Baxter operator~\mcite{Fard} with motivation from modified classical Yang-Baxter equation on Lie algebras~\cite{Sha}. As an application, using the composition-diamond lemma, we obtain the free objects in the category of modified Rota-Baxter algebras.

Putting Rota's Classification Problem in the contexts of rewriting systems and Gr\"{o}bner-Shirshov bases reveals the broad implication of Rota's Classification Problem and provides a framework that the problem might be further investigated and eventually resolved. The connection with Gr\"obner-Shirshov bases in operated algebras is comparable in spirit to Burchburger's Gr\"obner basis theory for commutative algebras and Bergman's analogue for algebras~\cite{Ber}.

\smallskip
\noindent
{\bf Notations}. Throughout this paper, we fix a field \bfk. Denote by $\nbfk:= \bfk\setminus \{0\}$ the subset of nonzero elements of \bfk. We denote the \bfk-span of a set $Y$ by $\bfk Y$. By an algebra, we mean an associative unitary \bfk-algebra. For any set $Y$, let $M(Y)$ denote the free monoid on $Y$ with identity $1$ and $S(Y)$ the free semigroup on $Y$.

\section{Reformulations of Rota's Classification Problem}\mlabel{sec:sgood}
\mlabel{sec:back}
In this section, we first recall some background on operated polynomial identities. We then introduce the concepts of \cs and \pcs systems of OPIs, and \gs and \pgs systems of OPIs, our main objects of study in this paper. We then reformulate Rota's Classification Problem in terms of these concepts.

\subsection{Operated polynomial identities}
\label{sec:freemonoid}

The concept of algebras with linear operators was first introduced by A.\,G. Kurosh~\cite{Ku} under the name of $\Omega$-algebras. It is called an operated algebra in~\cite{Gop} where the construction of free operated algebras was obtained. See also~\cite{BCQ,GSZ}. We briefly recall the construction and refer the reader to the above references for details.

\begin{defn}
{\rm An {\bf operated monoid} (resp. {\bf operated $\bfk$-algebra}) is a monoid (resp. $\bfk$-algebra) $U$ together with a map (resp. $\bfk$-linear map) $P_U: U\to U$.
A morphism from an operated monoid\, (resp. $\bfk$-algebra) $(U, P_U)$ to an operated monoid (resp. $\bfk$-algebra) $(V,P_V)$ is a monoid (resp. $\bfk$-algebra, resp. $\bfk$-module) homomorphism $f :U\to V$ such that $f \circ P_U= P_V \circ f$. } \label{de:mapset}
\end{defn}

Let $X$ be a given set. We will construct the free operated monoid over $X$.
The construction proceeds via the finite stages $\frakM_n(X)$ recursively defined as follows.
The initial stage is $\frakM_0(X) := M(X)$ and $\frakM_1(X) := M(X \cup \lc \frakM_0(X)\rc)$,
where $\lc \frakM_0(X)\rc:= \{ \lc u\rc \mid u\in \frakM_0(X)\}$ is a disjoint copy of $\frakM_0(X)$.
The inclusion $X\hookrightarrow X \cup \lc \frakM_0\rc $ induces a monomorphism
$$i_{0}:\mapmonoid_0(X) = M(X) \hookrightarrow \mapmonoid_1(X) = M(X \cup \lc \frakM_0\rc  )$$
of monoids through which we identify $\mapmonoid_0(X) $ with its image in $\mapmonoid_1(X)$.

For~$n\geq 2$, assume inductively that
$\frakM_{ n-1}(X)$ has been defined and the embedding
$$i_{n-2,n-1}\colon  \frakM_{ n-2}(X) \hookrightarrow \frakM_{ n-1}(X)$$
has been obtained. Then we define
\begin{equation*}
 \label{eq:frakn}
 \frakM_{ n}(X) := M \big( X\cup\lc\frakM_{n-1}(X) \rc\big).
\end{equation*}
Since~$\frakM_{ n-1}(X) = M \big(X\cup \lc\frakM_{ n-2}(X) \rc\big)$ is a free monoid,
the injection
$$  \lc\frakM_{ n-2}(X) \rc \hookrightarrow
    \lc \frakM_{ n-1}(X) \rc $$
induces a monoid embedding
\begin{equation*}
    \frakM_{ n-1}(X) = M \big( X\cup \lc\frakM_{n-2}(X) \rc \big)
 \hookrightarrow
    \frakM_{ n}(X) = M\big( X\cup\lc\frakM_{n-1}(X) \rc \big).
\end{equation*}
Finally we define the monoid
$$ \frakM (X):=\bigcup_{n\geq 0}\frakM_{ n}(X),$$
whose elements are called {\bf bracketed words} or
{\bf  bracketed monomials on $X$}.

Let $\kdot\frakM(X)$ be the free $\bfk$-module spanned by
$\frakM(X)$. The
multiplication on $\frakM(X)$ extends by linearity to turn the
$\bfk$-module $\kdot\frakM(X)$ into a $\bfk$-algebra.
Furthermore, we extend the operator
$\lc\ \rc: \frakM(X) \to \frakM(X), \ w\mapsto \lc w\rc$
to an operator $P$ on $\bfk\frakM(X)$ by linearity,
turning the $\bfk$-algebra $\bfk\frakM(X)$ into an operated
$\bfk$-algebra.

\begin{lemma}{\rm (\cite[Corollary~3.7]{Gop})}
Let $i_X:X \to \frakM(X)$ and $j_X: \frakM(X) \to \bfk\mapm{X}$ be the natural embeddings. Then, with the notations above,
\begin{enumerate}
\item
the triple $(\frakM(X),\lc\ \rc, i_X)$ is the free operated monoid on $X$; and
\label{it:mapsetm}
\item
the triple $(\bfk\mapm{X},P, j_X\circ i_X)$ is the free operated $\bfk$-algebra
on $X$. \label{it:mapalgsg}
\end{enumerate}
\label{pp:freetm}
\end{lemma}

\begin{defn}
Let $\phi(x_1,\ldots,x_k)\in \bfk\mapm{X}$ with $k\geq 1$ and $x_1, \ldots, x_k\in X$. We call
$\phi(x_1,\ldots,x_k) = 0$ (or simply $\phi(x_1,\ldots,x_k)$) an
{\bf  operated polynomial identity (OPI)}.
\end{defn}
Let $\phi= \phi(x_1,\ldots,x_k)\in \bfk\mapm{X}$ be an OPI. For any operated algebra $(R,P)$ and any map $\theta: x_i \mapsto r_i, i=1, \ldots, k$, using the universal property of $\bfk\mapm{x_1, \ldots, x_k}$ as a free operated algebra on $\{x_1,\cdots,x_k\}$, there is a unique morphism $\free{\theta}:\bfk\mapm{x_1,\ldots,x_k}\to R$ of operated algebras that extends the map $\theta$. We use the notation
\begin{equation*}
\phi(r_1,\ldots,r_k):= \free{\theta}(\phi(x_1,\ldots,x_k))
\label{eq:phibar}
\end{equation*}
for the corresponding {\bf evaluation} or {\bf substitution} of $\phi(x_1,\ldots,x_k)$ at the point $(r_1,\ldots,r_k)$. Informally, this is the element of $R$ obtained from $\phi$ upon
replacing every $x_i$ by $r_i$, $1\leq i\leq k$ and the operator $\lc\ \rc$ by $P$.
\begin{defn}
With the above notations, we say that $\phi(x_1,\ldots,x_k)=0$ (or simply $\phi(x_1,\ldots,x_k)$) is an {\bf OPI satisfied by $(R,P)$} if
$$\phi(r_1,\ldots,r_k)=0 \quad \text{for all } r_1,\ldots,r_k\in R.$$
In this case, we call $(R,P)$ (or simply $R$) a {\bf  $\phi$-algebra} and $P$ a {\bf $\phi$-operator}.
More generally, For a subset $\Phi\subseteq \bfk\mapm{X}$, we call $R$ (resp. $P$) a {\bf $\Phi$-algebra} (resp. {\bf $\Phi$-operator}) if $R$ (resp. $P$) is a $\phi$-algebra (resp. $\phi$-operator) for each $\phi\in \Phi$.
\label{de:phialg}
\end{defn}

For example, when $\phi=\lc x_1 x_2\rc -\lc x_1\rc x_2 -x_1\lc x_2\rc$ (resp. $\phi=\lc x_1\rc \lc x_2\rc-\lc x_1\lc x_2\rc\rc -\lc \lc x_1\rc x_2\rc -\lambda \lc x_1x_2\rc$), a $\phi$-algebra is simply a differential algebra (resp. a Rota-Baxter algebra of weight $\lambda$). When $\phi = x_1x_2-x_2x_1$, a $\phi$-algebra is a commutative algebra.

For $S\subseteq R$, the {\bf operated ideal} $\Id(S)$ of~$R$ generated by $S$ is defined to be the smallest operated ideal of $R$ containing $S$.
For $\Phi\subseteq \bfk\mapm{X}$ and a set $Z$, let $S_\Phi(Z) \subseteq \bfk\mapm{Z}$ denote the substitution set
\begin{equation}
S_\Phi(Z) := \{\,{\phi}(u_1,\ldots,u_k) \mid u_1,\ldots,u_k\in \mapm{Z}, \phi(x_1,\ldots,x_k)\in \Phi\,\}.
\mlabel{eq:genphi}
\end{equation}
The following well-known result exhibits the existence of a free $\Phi$-algebra whose explicit construction will be explored in this paper.

\begin{prop} {\rm (\cite[Proposition~1.3.6]{Coh})} Let $X$ be a set and $\Phi\subseteq \bfk\frakM(X)$ a system of OPIs. Then for a set $Z$, the quotient operated algebra $\bfk\mapm{Z}/\Id(S_\Phi(Z))$ is the free $\Phi$-algebra on $Z$.
\mlabel{pp:frpio}
\end{prop}

\subsection{Rota's Classification Problem via rewriting systems} \mlabel{ssec:pre}
As preparation, we recall concepts on \term-rewriting systems from \mcite{BN,GGSZ}.

\begin{defn}{\rm
Let $V$ be a $\bfk$-space with a given $\bfk$-basis $W$.
\begin{enumerate}
\item For $f=\sum\limits_{w\in W}c_w w \in V$ with $c_w\in \bfk$, the {\bf  support} $\Supp(f)$ of $f$ is the set $\{w\in W\,|\,c_w\neq 0\}$. As convention, we take $\Supp(0) = \emptyset$.
\item Let $f,g\in V$. We use $f \dps g$ to indicate the property that $\Supp(f) \cap \Supp(g) = \emptyset$. If this is the case, we say $f \dps g$ is a {\bf  direct sum} of $f$ and $g$, and use $f\dps g$ also for the sum $f+ g$.

\item For $f \in V$ and $w \in \Supp(f)$ with the coefficient $c_w$, write $R_w(f) := c_w w -f \in V$. So $f = c_w w \dps (-R_w(f))$.
\end{enumerate}
\mlabel{def:dps}
}\end{defn}

\begin{defn}{\rm Let $V$ be a $\bfk$-space with a $\bfk$-basis $W$.
\begin{enumerate}
\item  A {\bf  \term-rewriting system $\Pi$ on $V$ with respect to $W$} is a binary relation $\Pi \subseteq W \times V$. An element $(t,v)\in \Pi$ is called a (\term-) rewriting rule of $\Pi$, denoted by $t\to v$.

\item The \term-rewriting system $\Pi$ is called {\bf  \simple with respect to $W$} if $t \dps v$ for all $t\to v\in \Pi$.

\item If $f = c_t t\dps (-R_t(f))\in V$, using the rewriting rule $t\to v$, we get a new element $g:= c_t v - R_t(f) \in V$, called a {\bf one-step rewriting} of $f$ and denoted $f \to_\Pi g$ or $\tvarrow{f}{g}{\Pi}$.  \label{item:Trule}

\item The reflexive-transitive closure of $\rightarrow_\Pi$ (as a binary relation on $V$) is denoted by $\astarrow_\Pi$ and, if $f \astarrow_\Pi g$, we say {\bf  $f$ rewrites to $g$ with respect to $\Pi$}. \label{item:rtcl}

\item Two elements  $f, g \in V$ are {\bf  joinable} if there exists $h \in V$ such that $f \astarrow_\Pi h$ and $g \astarrow_\Pi h$; we denote this by $f \downarrow_\Pi g$.

\item An element $f\in V$ is {\bf a normal form} if no more rules from $\Pi$ can apply, more precisely, if $\Supp(f)\cap \mathrm{Dom}(\Pi)=\emptyset$ where $\mathrm{Dom}(\Pi)$ is the domain of $\Pi\subseteq W\times V$.
\end{enumerate}
\label{def:ARSbasics}}
\end{defn}

The crucial point of Item~(\mref{item:Trule}) in Definition~\mref{def:ARSbasics} is that, in order to apply a rewriting rule $t\to v$ to $f$, one
must firstly express $f$ as the direct sum $f=c_t t\dps (-R_t(f))$.
The following definitions are adapted from abstract rewriting systems \mcite{BN, Oh}.

\begin{defn}
A \term-rewriting system $\Pi$ on $V$ is called
 \begin{enumerate}
 \item {\bf  terminating} if there is no infinite chain of one-step rewriting \vspace{-3pt}$$f_0 \rightarrow_\Pi f_1 \rightarrow_\Pi f_2 \cdots \quad.\vspace{-3pt}$$
\item {\bf  confluent} (resp. {\bf  locally confluent}) if every fork (resp. local fork) is joinable.
\item {\bf  convergent} if it is both terminating and confluent.
\end{enumerate}
\label{def:ARS}
\end{defn}

Given a system of OPIs, we can associate it with a \term-rewriting system.
For this, we need the following concept.

\begin{defn}
Let $Z$ be a set, $\star$ a symbol not in $Z$ and $Z^\star = Z \cup \{\star\}$.
\begin{enumerate}
\item By a {\bf  $\star$-bracketed word} on $Z$, we mean any bracketed word in $\mapm{Z^\star}$ with exactly one occurrence of $\star$, counting multiplicities. The set of all $\star$-bracketed words on $Z$ is denoted by $\frakM^{\star}(Z)$.
\item For $q\in \frakM^\star(Z)$ and $u \in  \frakM({Z})$, we define $q|_{\star \mapsto u}$ to be the bracketed word on $Z$ obtained by replacing the symbol $\star$ in $q$ by $u$.

\item For $q\in \frakM^\star(Z)$ and $s=\sum_i c_i u_i\in \bfk\frakM{(Z)}$, where $c_i\in\bfk$ and $u_i\in \frakM{(Z)}$, we define
\begin{equation*}
 q|_{s}:=\sum_i c_i q|_{u_i}\,. \vspace{-5pt}
\end{equation*}

\item A bracketed word $u \in \frakM(Z)$ is a {\bf  subword} of another
bracketed word $w \in \frakM(Z)$ if $w  = q|_u$ for some $q \in
\frakM^\star(Z)$.
\end{enumerate}
More generally, let $\star_1,\ldots,\star_k$ be distinct symbols not in
$Z$ and set $Z^{\star k} := Z\cup \{\star_1, \dots, \star_k \}$, $k\geq 1$.
\begin{enumerate}
\setcounter{enumi}{4}
\item We define an {\bf $(\star_1, \dots,
    \star_k)$-bracket word on $Z$} to be an expression in
  $\frakM(Z^{\star k})$ with exactly one occurrence of each of
  $\star_i$, $1\leq i\leq k$. The set of all
  $(\star_1, \dots, \star_k)$-bracket words on $Z$
  is denoted by $\frakM^{\star k}(Z)$.

\item For $q\in \frakM^{\star k}(Z)$ and $u_1,\dots, u_k\in
  \bfk \frakM^{\star k}(Z)$, we define
$$\stars{q}{u_1,\dots, u_k} := \stars{q}{\star_1 \mapsto u_1, \ldots, \star_k \mapsto u_k}$$
to be the element of $\bfk\frakM(Z)$ obtained from $q$ when the letter $\star_i$, $1\leq i\leq k$, in $q$ is replaced by $u_i$.
\end{enumerate}
 \mlabel{def:starbw}
\end{defn}

\begin{defn}
Let $Z$ be a set and $S\subseteq \bfk \frakM(Z)$.
\begin{enumerate}
\item
Let $s\in \bfk\frakM(Z)$ and fix a monomial $\fix{s}$
of $s$, called an {\bf orientation} of $s$. The  {\bf monicization of $s$ with respect to} $\fix{s}$ is replacing $s$ by its quotient over the coefficient of $\fix{s}$, making $s$ monic if $\fix{s}$ is taken as the leading term. When this is done for each $s$ in a subset $S$ of $\bfk\frakM(Z)$, then we call $S$ {\bf monicized with respect to the orientation} $\fix{S}:=\{\fix{s}\mid s\in S\}$.
\item Let $S\subseteq \bfk\frakM(Z)$ with a given orientation $\fix{S}:=\{\fix{s}\mid s\in S\}$.
We can write $s = \fix{s} \dps (-R(s))$. Define a term-rewriting system on $\bfk\mapm{Z}$ by
\begin{equation*}
\Pi_{S}:=\Pi_{S,\fix{S}}(Z) := \{\, q|_{\fix{s}}\to  q|_{R(s)} \mid  s \in S, ~q \in \frakM^\star(Z)\,\} \subseteq \mapm{Z} \times \bfk\mapm{Z}.
\end{equation*}
We call $\Pi_S$ the {\bf \term-rewriting system} associated to $S$ with respect to $\fix{S}=\{\fix{s} \mid s\in S\}$.
\label{it:fix1}
\item Let $\Phi \subseteq \bfk \mapm{X}$ be a system of OPIs. For a set $Z$, let
$$\fix{S}_{{\Phi}}(Z):=\{\fix{\phi(\fraku)}\mid\phi(\fraku)\in S_\Phi(Z) \},$$
be an orientation of the set $S_\Phi(Z)$ in Eq.~(\mref{eq:genphi}).
We call the resulting rewriting system
\begin{equation*}
\Pi_{S_\Phi(Z)}:= \Pi_{S_\Phi(Z),\,\fix{S}_\Phi(Z)}:=\{\,q|_{\fix{\phi(\fraku)}}\to q|_{\re{\phi(\fraku)}} \mid q \in \frakM^\star(Z), \phi(\fraku)\in S_\Phi(Z) \} \subseteq \mapm{Z}\times \bfk\mapm{Z}
%
\end{equation*}
the {\bf term-rewriting system with respect to $\fix{S}_\Phi(Z)$}.
In particular, if $\Phi = \{\phi\}$, we get a \term-rewriting system associated to $\phi$ with respect to $\fix{S}_\phi(Z):= \{\fix{\phi(\fraku)} \mid \phi(\fraku)\in S_\phi(Z) \}$
\begin{equation*}
\Pi_{S_\phi(Z)}:= \Pi_{S_\phi(Z),\,\fix{S}_{{\phi}}(Z)}:=\{\,q|_{\fix{\phi(\fraku)}}\to q|_{\re{\phi(\fraku)}} \mid q \in \frakM^\star(Z), \phi(\fraku)\in S_\phi(Z) \} \subseteq \mapm{Z}\times \bfk\mapm{Z}.
%
\end{equation*}
\label{it:ori3}
\end{enumerate}
\label{def:redsys}
\end{defn}

For notational clarify, we will often abbreviate $\to_{\Pi_{S_\phi(Z)}}$ (resp. $\astarrow_{\Pi_{S_\phi(Z)}}$, resp. $\downarrow_{\Pi_{S_\phi(Z)}}$) to $\to_\phi$ (resp. $\astarrow_\phi$, resp. $\downarrow_{\phi}$).

\begin{defn}
Let $X$ be a set and $\Phi\subseteq \bfk \frakM(X)$ a system of OPIs. Let $Z$ be a set and $\Pi_{S_\Phi(Z)} = \Pi_{S_\Phi(Z),\, \fix{S}_\Phi(Z)}$ a \term-rewriting system with respect to an orientation $\fix{S}_\Phi(Z)$ of $S_\Phi(Z)$.
\begin{enumerate}
\item We call $\Phi$ {\bf convergent on $Z$ with respect to $\fix{S}_\Phi(Z)$} if $\Pi_{S_\Phi(Z)} = \Pi_{S_\Phi(Z),\, \fix{S}_\Phi(Z)}$ is convergent.
\mlabel{it:pwrtp}
\item We call $\Phi$ {\bf potentially convergent on $Z$ with respect to $\fix{S}_\Phi(Z)$} if,
there is a superset $\Phi'\subseteq \bfk\frakM(Z)$ of $\Phi$ with $\Id(S_\Phi(Z)) = \Id(S_{\Phi'}(Z))$
 and an orientation $\fix{S}_{\Phi'}(Z):=\{ \fix{\phi(\fraku)} \mid \phi(\fraku)\in S_{\Phi'}(Z) \}$
containing $\fix{S}_\Phi(Z)$, such that $\Pi_{S_{\Phi'}(Z),\, \fix{S}_{\Phi'}(Z)}$ is convergent.
\end{enumerate}
\mlabel{def:cwpp}
\end{defn}

\begin{defn}
Let $X$ be a set, and let $\Phi\subseteq \bfk \frakM(X)$ be a system of OPIs.
\begin{enumerate}
\item We call $\Phi$ {\bf convergent} (resp. {\bf potentially convergent}) if, for each set $Z$, there is an orientation $\fix{S}_\Phi(Z)$ such that $\Phi$ is convergent (resp.  potentially convergent) on $Z$ with respect to $\fix{S}_\Phi(Z)$.
\item A $\Phi$-operator $P$ is called {\bf convergent} (resp. {\bf potentially convergent}) if $\Phi$ is so.
\end{enumerate}
\end{defn}

We can now interpret Rota's Classification Problem in terms of rewriting systems.
\begin{problem}
{\bf (Rota's Classification Problem via rewriting systems)} Determine all convergent and potentially
convergent systems of OPIs.
\label{pr:rotare}
\end{problem}

The well-known (two-sided) averaging operator $P$~(see \mcite{GP} for example) satisfies
$$ P(x_1)P(x_2)=P(P(x_1)x_2)=P(x_1P(x_2))$$
and hence is defined by the system of OPIs
\begin{equation}
\begin{aligned}
\phi_1(x_1, x_2) &= \lc x_1\rc\lc x_2\rc - \lc \lc x_1\rc x_2\rc,\\
\phi_2(x_1, x_2) &= \lc x_1\lc x_2\rc\rc - \lc \lc x_1\rc x_2\rc.
\end{aligned}
\mlabel{eq:phi12}
\end{equation}

\begin{prop}
The system of OPIs for the (two-sided) averaging operator is not convergent.
\mlabel{pp:anotsg}
\end{prop}

As we will see in Remark~\ref{re:apconv}, this system of OPIs is \pcs.

\begin{proof}
Let $Z=\{z_1,z_2\}$, $w = \lc \lc z_1\rc \lc z_2\rc\rc \in \frakM(Z)$ and $\Phi = \{\phi_1, \phi_2\}$.
Write
$$ \phi_1 = \phi_1(z_1, z_2)\,\text{ and }\, \phi_2=\phi_2(z_1, z_2).$$
According to the choice of orientations $\fix{\phi_1}$ and $\fix{\phi_2}$ of $\phi_1$ and $\phi_2$, we have the following four cases.
\smallskip

\noindent
{\bf Case 1.} $\fix{\phi_1} = \lc z_1\rc\lc z_2\rc$ and $\fix{\phi_2} = \lc z_1\lc z_2\rc\rc$. Then Eq.~(\mref{eq:phi12}) induces two rewriting rules
\begin{equation*}
\begin{aligned}
\lc z_1\rc\lc z_2\rc \to_{\phi_1}  \lc \lc z_1\rc z_2\rc\,\text{ and }\,
\lc z_1\lc z_2\rc\rc \to_{\phi_2}  \lc \lc z_1\rc z_2\rc.
\end{aligned}
\end{equation*}
We have
$$w=\lc \lc z_1\rc \lc z_2\rc\rc \to_{\phi_1}  \lc \lc\lc z_1\rc z_2\rc\rc \,\text{ and }\,
w= \lc \lc z_1\rc \lc z_2\rc\rc \to_{\phi_2} \lc \lc \lc z_1\rc \rc z_2\rc.$$
Since $\lc \lc\lc z_1\rc z_2\rc\rc$ and  $\lc \lc \lc z_1\rc \rc z_2\rc$ are different normal forms, $\Pi_{S_\Phi(Z)}$ is not confluent.
\smallskip

\noindent
{\bf Case 2.} $\fix{\phi_1} = \lc z_1\rc\lc z_2\rc$ and
$\fix{\phi_2} = \lc \lc z_1\rc z_2\rc$. Then Eq.~(\mref{eq:phi12}) induces two rewriting rules
\begin{equation*}
\begin{aligned}
\lc z_1\rc\lc z_2\rc \to_{\phi_1}  \lc \lc z_1\rc z_2\rc\,\text{ and }\,
 \lc \lc z_1\rc z_2\rc \to_{\phi_2} \lc z_1\lc z_2\rc\rc .
\end{aligned}
\end{equation*}
We have
$$w=\lc \lc z_1\rc \lc z_2\rc\rc \to_{\phi_1}  \lc \lc\lc z_1\rc z_2\rc\rc \to_{\phi_2}
\lc \lc z_1\lc z_2\rc \rc\rc
\,\text{ and }\,
w=\lc \lc z_1\rc \lc z_2\rc\rc \to_{\phi_2} \lc z_1\lc \lc z_2\rc \rc\rc.$$
Again, since $\lc \lc z_1\lc z_2\rc \rc\rc$ and  $\lc z_1\lc \lc z_2\rc \rc\rc$ are different normal forms, $\Pi_{S_\Phi(Z)}$ is not confluent.
\smallskip

\noindent
{\bf Case 3.} $\fix{\phi_1} = \lc\lc z_1\rc z_2\rc$ and $\fix{\phi_2} = \lc z_1\lc z_2\rc\rc$. Then Eq.~(\mref{eq:phi12}) induces two rewriting rules
\begin{equation*}
\begin{aligned}
 \lc \lc z_1\rc z_2\rc \to_{\phi_1} \lc z_1\rc\lc z_2\rc\,\text{ and }\,
\lc z_1\lc z_2\rc\rc \to_{\phi_2}  \lc \lc z_1\rc z_2\rc.
\end{aligned}
\end{equation*}
We have
$$w=\lc \lc z_1\rc \lc z_2\rc\rc \to_{\phi_1}  \lc z_1\rc \lc \lc z_2\rc\rc \,\text{ and }\,
w=\lc \lc z_1\rc \lc z_2\rc\rc \to_{\phi_2} \lc \lc \lc z_1\rc \rc z_2\rc \to_{\phi_1}
 \lc \lc z_1\rc \rc \lc z_2\rc.$$
Since $ \lc z_1\rc \lc \lc z_2\rc\rc $ and $ \lc \lc z_1\rc \rc \lc z_2\rc$ are different normal forms, $\Pi_{S_\Phi(Z)}$ is not confluent.
\smallskip

\noindent
{\bf Case 4.} $\fix{\phi_1} = \lc \lc z_1\rc z_2\rc$ and
$\fix{\phi_2} = \lc \lc z_1\rc z_2\rc$. Then Eq.~(\mref{eq:phi12}) induces two rewriting rules
\begin{equation*}
\begin{aligned}
 \lc \lc z_1\rc z_2\rc \to_{\phi_1} \lc z_1\rc\lc z_2\rc\,\text{ and }\,
 \lc \lc z_1\rc z_2\rc \to_{\phi_2} \lc z_1\lc z_2\rc\rc .
\end{aligned}
\end{equation*}
We have
$$\lc \lc z_1\rc \lc z_2\rc\rc \to_{\phi_1}  \lc z_1\rc  \lc \lc z_2\rc\rc
\,\text{ and }\,
\lc \lc z_1\rc \lc z_2\rc\rc \to_{\phi_2} \lc z_1\lc \lc z_2\rc\rc\rc.$$
Again, since $\lc z_1\rc  \lc \lc z_2\rc\rc$ and  $\lc z_1\lc \lc z_2\rc\rc\rc$ are different normal forms, $\Pi_{S_\Phi(Z)}$ is not confluent.

In summary, for the set $Z=\{z_1, z_2\}$, there is no $\Pi_{S_\Phi(Z)}$ such that $\Phi$ is confluent
on $Z$ with respect to $\fix{S}_\Phi(Z)$. So $\Phi$ is not convergent.
\end{proof}

\subsection{Rota's Classification Problem via Gr\"{o}bner-Shirshov bases} \mlabel{ssec:pgs}
In this subsection, we give the definitions of \gs and \pgs systems of OPIs. Let us first recall some background on Gr\"obner-Shirshov bases. See~\cite{BCQ,GSZ} for further details.

\begin{defn}
Let $Z$ be a set, $\leq$ a linear order on $\frakM(Z)$ and $f \in \bfk\mapm{Z}$.
\begin{enumerate}
\item
Let $f \notin \bfk$. The {\bf leading monomial} of $f$, denoted by $\bar{f}$, is the largest monomial appearing in $f$. The {\bf  leading coefficient of $f$}, denoted by $c_f$, is the coefficient of $\bar{f}$ in $f$.
We call $f$ {\bf  monic} with respect to $\leq$ if $c_f=1$.
\label{item:monic}

\item If $f \in \bfk$ (including the case $f = 0$), we define the {\bf  leading monomial of $f$} to be $1$ and the {\bf  leading coefficient of $f$} to be $c_f= f$.\label{item:scalar}

\item A subset $S\subseteq \bfk\mapm{Z}$ is called {\bf monicized with respect to $\leq$} if each element of $S$ is replaced by its quotient over the coefficient of its leading monomial, and hence is monic.
\end{enumerate}
\label{def:irrS}
\end{defn}

\begin{defn} Let $Z$ be a set.
A {\bf  monomial order on $\frakM(Z)$} is a
well-order $\leq$ on $\frakM(Z)$ such that
\begin{equation}
u < v \Longrightarrow q|_u < q|_v \quad  \text{for all } u, v \in \frakM(Z) \text{ and } q \in \frakM^{\star}(Z).
\label{eq:morder}
\end{equation}
We denote $u < v$ if $u\leq v$ but $u\neq v$.
\label{de:morder}
\end{defn}

Since $\leq$ is a well-order, it follows from Eq.\,(\ref{eq:morder}) that $1 \leq u$ and $u < \lc u \rc$ for all $u \in \frakM(Z)$.

\begin{remark}
If there is a linear order $\leq$ on $\frakM(Z)$, then in Definition~\mref{def:redsys}, we can take $\fix{s}$ as the leading monomial $\lbar{s}$ of $s$ with respect to $\leq$. We call  $\lbar{S}:= \{\lbar{s} \mid s\in S\}$ the {\bf orientation from} $\leq$, and
\begin{equation*}
\Pi_{S} = \Pi_{S, \lbar{S}}=\{\, q|_{\lbar{s}}\to  q|_{R(s)} \mid  s \in S, ~q \in \frakM^\star(Z)\,\} \subseteq \mapm{Z} \times \bfk\mapm{Z}
\end{equation*}
the {\bf term-rewriting system from} $\leq$.
\end{remark}

Let $f\in \frakM(Z)$ with $f \ne 1$. Then $f$ can be uniquely written as a product $f_1\cdots f_n$, where $n\leq 1$ and $f_i\in Z\cup \lc \frakM(Z)\rc$ for
$1\leq i\leq n$. We call $n$ the {\bf  breadth} of $f$, denoted
by $|f|$. If $f = 1$, we define $|f| = 0$.

\begin{defn}
Let $\leq$ be a monomial order on $\frakM(Z)$ and $f, g \in \bfk\mapm{Z}$ be monic.
\begin{enumerate}
\item If there are $w, u, v\in \frakM(Z)$ such that $w=\bar{f}u=v\bar{g}$ with $\max\{\bre{\bar{f}},\bre{\bar{g}}\}<\bre{w}< \bre{\bar{f}}+\bre{\bar{g}}$, we call
$$(f,g)^{u,v}_w:=fu-v g$$
the {\bf  intersection composition of $f$ and $g$ with respect to $w$}.
\item  If there are $w\in \frakM(Z)$ and $q\in \frakM^{\star}(Z)$ such that $w =\bar{f}=q|_{\bar{g}}$,
we call
$$(f,g)^{q}_w:=f-q|_{g}$$ the {\bf including
composition of $f$ and $g$ with respect to $w$}.
\end{enumerate}
\end{defn}

\begin{defn}{\rm
Let $Z$ be a set and $\leq$ a monomial order on $\mapm{Z}$.
\begin{enumerate}
\item An element $f \in \bfk\mapm{Z}$ is called {\bf  trivial modulo $(S, w)$} if
$$f = \sum_ic_iq_i|_{s_i} \text{ with } q_i|_{\overline{s_i}}< w, \text{ where } c_i\in \bfk, q_i \in \frakM^{\star}(Z), s_i\in S.$$

\item Let $S\subseteq \bfk\mapm{Z}$. Then $S$ is called a {\bf  Gr\"{o}bner-Shirshov basis} in $\bfk\frakM(Z)$ with respect to $\leq$ if,
for all pairs $f,g\in S$ monicized with respect to $\leq$, every intersection composition of the form $(f,g)_w^{u,v}$
is trivial modulo $(S,w)$, and every including composition of the form $(f,g)_w^q$ is trivial modulo $(S,w)$.
\label{de:GS}
\end{enumerate}
}\end{defn}

By convention, the polynomial 0 is trivial modulo $(S, w)$ for any $S$ and $w$.
The Composition-Diamond Lemma is the corner stone of the theory of Gr\"{o}bner-Shirshov bases.

\begin{theorem} {\rm (Composition-Diamond Lemma \cite{BCQ,GSZ})}
Let $Z$ be a set, $\leq$ a monomial order on $\mapm{Z}$ and $S \subseteq \bfk\mapm{Z}$. Then the following conditions are equivalent.
\begin{enumerate}
\item $S $ is a Gr\"{o}bner-Shirshov basis in $\bfk\mapm{Z}$.
\label{it:cd1}

\item Let $\eta\!: \bfk\mapm{Z} \to \bfk\mapm{Z}/\Id(S)$ be the quotient homomorphism of\, $\bfk$-spaces. Denote
\begin{equation}
\irr(S):= \frakM(Z) \setminus \{ q|_{\lbar{s}} \mid s\in S \}.
\mlabel{eq:irrs}
\end{equation}
    As a $\bfk$-space, $\bfk\mapm{Z}=\bfk\Irr(S)\oplus \Id(S)$
and $\eta(\Irr(S))$ is a $\bfk$-basis of $\bfk\mapm{Z}/\Id(S)$.
 \label{it:cd4}
\end{enumerate}
\mlabel{thm:cdl}
\end{theorem}

By~\cite{GGSZ,GSZ}, differential type OPIs and Rota-Baxter type OPIs (See Section~\ref{ss:thmexam} for definitions), which comprise all the OPIs that motivated Rota to have posed his classification problem except the Reynolds OPI, can be characterized by possessing Gr\"obner-Shirshov bases. This prompts us to introduce the following notions.

\begin{defn}
Let $X$ be a set and $\Phi\subseteq \bfk\frakM(X)$ a system of OPIs. Let $Z$ be a set and $\leq$
a monomial order on $\frakM(Z)$.
\begin{enumerate}
\item  We call $\Phi$ {\bf \gs  on $Z$ with respect to} $\leq$ if $S_\Phi(Z)$ is a Gr\"{o}bner-Shirshov basis in $\bfk \frakM(Z)$ with respect to $\leq$.

\item We call $\Phi$ {\bf \pgs on $Z$ with respect to} $\leq$ if there is a superset $\Phi'\subseteq \bfk\frakM(X)$ of $\Phi$ such that $\Id(S_\Phi(Z)) = \Id(S_{\Phi'}(Z))$ and
$\Phi'$ is \gs on $Z$ with respect to $\leq$.
\mlabel{it:pgsord}

\end{enumerate}
\label{def:pgord}
\end{defn}

\begin{defn}
Let $X$ be a set and $\Phi\subseteq \bfk\frakM(X)$ a system of OPIs.
\begin{enumerate}
\item  We call $\Phi$ {\bf \gs } (resp. {\bf \pgs}) if, for each set $Z$, there is a monomial order $\leq$ on $\frakM(Z)$ such that $\Phi$ is \gs (resp. {\bf \pgs}) on $Z$ with respect to $\leq$.
\item A $\Phi$-operator $P$ is called {\bf \gs} (resp. {\bf \pgs}) if $\Phi$ is.
\end{enumerate}
\label{def:sggood}
\end{defn}

\begin{exam}
A differential type OPI~\cite{GSZ}, defining a differential type operator $d = \lc \,\rc$, is
$$\phi=\phi(x_1,x_2) = \lc x_1 x_2\rc -N(x_1, x_2),$$
with $N(x_1,x_2)\in \bfk\mapm{x_1,x_2}$ satisfying certain conditions, to be recalled in Example~\ref{ex:difftype}. By \cite[Theorem.~5.7]{GSZ}, $S_\phi(Z)$ is a Gr\"{o}bner-Shirshov basis of $\Id(S_\phi(Z))$ with respect to a monomial order. Hence a differential type OPI is \gs. This fact will be proved directly in Example~\ref{ex:difftype}.
\mlabel{ex:diffopi}
\end{exam}

\begin{exam}
A Rota-Baxter type OPI~\cite{GGSZ}, defining a Rota-Baxter type operator, is
$$\phi=\phi(x_1,x_2) = \lc x_1\rc \lc x_2\rc - \lc B(x_1, x_2)\rc,$$
with $B(x_1,x_2)\in \bfk\mapm{x_1,x_2}$ satisfying certain conditions detailed in Example~\ref{ex:rbtype}.
It was shown in~\cite[Corollary~3.13,\,Theorem~4.9]{GGSZ}, and again in~Example~\ref{ex:rbtype},
that $S_\phi(Z)$ is a Gr\"{o}bner-Shirshov basis of $\Id(S_\phi(Z))$ with respect to a monomial order. Hence a Rota-Baxter type OPI is \gs.
\mlabel{ex:rbopi}
\end{exam}

\begin{exam}
As shown below in Theorem~\mref{thm:mrb}, a modified Rota-Baxter type OPI is \gs.
By~\cite[Theorems.~2.41,\,3.10]{GZ}, the system of (two-sided) averaging OPIs defined in Eq.~(\mref{eq:phi12}) is \pgs.
\end{exam}

We now propose another reformulation of Rota's Classification Problem.

\begin{problem} {\bf (Rota's Classification Problem via Gr\"{o}bner-Shirshov bases)} Determine all \gs and \pgs systems of OPIs.
\label{pr:rota}
\end{problem}

\section{Relationship between reformulations of Rota's Classification Problem}
\mlabel{sec:sgoodrew}

In this section, we establish the relationship between reformulations of Rota's Classification Problem.

\subsection{Term-rewriting systems}
We recall some basic results from \mcite{GGSZ} for \term-rewriting systems.
We will need the following Newman's lemma on rewriting systems.

\begin{lemma}{\rm (\cite[Lemma~2.7.2]{BN})}
 A terminating rewriting system is confluent if and only if it is locally confluent. \label{lem:newman}
 \end{lemma}

The next results will also be used later.

\begin{lemma} {\rm (\cite[Proposition~2.18,\,Theorem~2.20]{GGSZ})}
Let $V$ be a $\bfk$-space with a given $\bfk$-basis $W$, and let
$\Pi$ be a simple \term-rewriting system on $V$ with respect to $W$.
\begin{enumerate}

\item $(f-g) \astarrow_\Pi 0$ implies $f \downarrow_\Pi g$  for all $f,g\in V$. \mlabel{it:join}

\item If\, $\Pi$ is confluent, then for all $f, g, h \in V$, $$f \downarrow_\Pi g,\  g \downarrow_\Pi h \implies f \downarrow_\Pi h.$$ \mlabel{it:trans}

\item If\, $\Pi$ is confluent, then
$$ f \downarrow_\Pi g,\  f' \downarrow_\Pi g' \implies (f + f') \downarrow_\Pi (g+g')\quad \forall f, g, f', g' \in V.$$ \mlabel{it:plus}

\item  If\, $\Pi$ is confluent, then, for all $m \geq 1$ and  $f_1, \dots, f_m, g_1, \dots, g_m \in V$, $$f_i \downarrow_\Pi g_i  \quad (1 \leq i \leq m), {\rm\ and\ } \sum_{i=1}^m g_i = 0 \implies \left(\sum_{i=1}^m f_i \right) \astarrow_\Pi 0.$$ \mlabel{it:sumzero}
\end{enumerate}
\mlabel{lem:red}
\end{lemma}

\begin{remark}
If $\Pi$ is confluent and $f \downarrow_\Pi g$, together with the fact $-g \downarrow_\Pi -g$, we get $f-g\astarrow_\Pi 0$ by Lemma~\mref{lem:red}(\mref{it:sumzero}).
\mlabel{re:tozero}
\end{remark}

The following is a stronger condition than locally confluence.

\begin{defn}
Let $V$ be a $\bfk$-spaces with a $\bfk$-basis $W$ and let $\Pi$ be a \simple term-rewriting system on $V$ with respect to $W$.
\begin{enumerate}
\item  \emph{ A local base-fork} is a fork $(k t \to_\Pi k v_1, k t \to_\Pi k v_2)$, where $k \in \nbfk$ and $t\to v_1, t\to v_2 \in \Pi$. The \term-rewriting system $\Pi$ is \emph{locally base-confluent} if for every local base-fork $(k t \to_\Pi k v_1, k t \to_\Pi k v_2)$, we have $k (v_1 - v_2) \astarrow_\Pi 0$.

\item We say that $\Pi$ is \emph{ compatible with a linear order $\leq$ on $W$} if $\bar{v} < t$ for each $t\to v \in \Pi$.
    \mlabel{it:pcomo}
\end{enumerate}
\mlabel{defn:lbcom}
\end{defn}

\begin{lemma}{\rm (\cite[Lemma~2.22]{GGSZ})} Let $V$ be a $\bfk$-space with a $\bfk$-basis $W$ and let $\Pi$ be a term-rewriting system on $V$ which is compatible with a well order $\leq$ on $W$. If $\Pi$ is locally base-confluent, then it is locally confluent.
\mlabel{lem:ltcon}
\end{lemma}

\begin{lemma}
Let $V$ be a $\bfk$-space with a given $\bfk$-basis $W$, and let
$\Pi$ be a simple \term-rewriting system on $V$ with respect to $W$. Let $f, g \in V$. If $f\astarrow_\Pi g$, then $kf\astarrow_\Pi kg $ for any $k\in \bfk$.
\mlabel{lem:cons}
\end{lemma}

\begin{proof}
If $f=g$ or $k=0$, then $kf =kg$ and $kf \astarrow_\Pi kg$. Suppose $f\neq g$ and $k\neq 0$. Let $n\geq 1$ be the minimum step that $f$ rewrites to $g$ and
$$f=:f_0\tpi f_1\tpi \cdots \tpi f_n:=g.$$
We prove the result by induction on $n$. If $n=1$, we may write
$$f = ct\dps (-R_t(f))\, \text{ and }\, g = cv - R_t(f),$$
where $c\in \nbfk$ and $t\to v\in \Pi$. Then
$$kf = kct\dps (-kR_t(f))\, \text{ and }\, kg = kcv - kR_t(f).$$
Since $k,c\neq 0$ and $\bfk$ is a field by our hypothesis, $kc\neq 0$ and so $kf \tpi kg$.
Assume that the result is true for $n \leq m$ and consider the case of $n=m+1$. Then by the induction hypothesis, we have $kf_0 \astarrow_\Pi kf_1$ and $kf_1\astarrow_\Pi kf_n$. Hence by the transitivity of $\astarrow_\Pi$ we have $kf = kf_0 \astarrow_\Pi kf_n = kg$, as required.
\end{proof}

The following concepts are adapted from general abstract rewriting
systems~\cite[Definition~1.1.6]{BKVT}.

\begin{defn}
Let $V$ be a $\bfk$-spaces with a $\bfk$-basis $W$ and let $\Pi$ be a \simple term-rewriting system on $V$ with respect to $W$. Let $Y\subseteq W$ and $\Pi_{\bfk Y} \subseteq Y\times \bfk Y$. We call $\Pi_{\bfk Y}$ a {\bf sub-term-rewriting system} of $\Pi$ on $\bfk Y$ with respect to $Y$, denoted by $\Pi_{\bfk Y} \leq \Pi$, if
\begin{enumerate}

\item $\Pi_{\bfk Y}$ is the restriction of $\Pi$, i.e., for any $f,g\in \bfk Y$, $f\to_{\Pi_{\bfk Y}} g \Leftrightarrow f\to_{\Pi} g$. \mlabel{it:restr}

\item $\bfk Y$ is closed under $\Pi$, i.e., for any $f\in \bfk Y$ and any $g\in V$,
$f\to_{\Pi} g$ implies $g\in \bfk Y$. \mlabel{it:closed}
\end{enumerate}
\mlabel{defn:subre}
\end{defn}

The following result characterizes the sub-term-rewriting system when $\Pi_{\bfk Y} =\Pi \cap (Y\times \bfk Y)$.

\begin{prop}
Let $V$ be a $\bfk$-space with a $\bfk$-basis $W$ and let $\Pi$ be a \simple term-rewriting system on $V$ with respect to $W$. Let $Y\subseteq W$ and $\Pi_{\bfk Y} :=\Pi \cap (Y\times \bfk Y)$. Then $\Pi_{\bfk Y}$ is a sub-term-rewriting system of $\Pi$ on $\bfk Y$ with respect to $Y$ if and only if $\bfk Y$ is closed under $\Pi$ in the sense of Definition~\mref{defn:subre}({\mref{it:closed}}).
\mlabel{pp:subre}
\end{prop}

\begin{proof}
($\Rightarrow$) This direction follows from Definition~\mref{defn:subre}.

($\Leftarrow$) With Item~(\mref{it:closed}) of Definition~\mref{defn:subre} being our hypothesis, we only need to show that Item~(\mref{it:restr}) is valid, that is, $\Pi_{\bfk Y}$ is the restriction of $\Pi$ to $\bfk Y$.
Let $f,g\in \bfk Y$ with $f\to_{\Pi_{\bfk Y}} g$. Since $\Pi_{\bfk Y} =\Pi \cap (Y\times \bfk Y) \subseteq \Pi$, we have $f\to_{\Pi} g$. Conversely, suppose $f\to_{\Pi} g$. Write $f = ct\dps f_1$ and $g = cv+ f_1$, where $c\in \nbfk$, $t\in Y$, $v\in V$, $f_1\in \bfk Y$ and $t\to v\in \Pi$.
Since $g\in \bfk Y$ and $f_1\in \bfk Y$, we have $cv\in \bfk Y$. Since $W$ is a \bfk-basis of $V$ and $Y\subseteq W$, we may write $$v = \sum_i c_i y_i + \sum_j d_j x_j,\,\text{ where } c_i,d_j\in \bfk, y_i\in Y, x_j\in W\setminus Y.$$
Then $$ cv = \sum_i cc_i y_i + \sum_j cd_j x_j \in \bfk Y\,\text{ and }\,  \sum_j cd_j x_j \in \bfk Y,$$
and so $cd_j = 0$ for each $j$. Since $\bfk$ is a field and $c\neq 0$, we get $d_j = 0$ for each $j$, that is, $v\in \bfk Y.$
Thus $t\to v\in \Pi_{\bfk Y}$ and so $ ct\dps f_1 \to_{\Pi_{\bfk Y}} cv + f_1$, as required.
\end{proof}

The \term-rewriting system $\Pi_{S_\Phi(Z)}$ from a monomial order is simple. To show this, we need the following fact.

\begin{lemma}
Let $Z$ be a set and $\leq$ a monomial order on $\frakM(Z)$. If $\stars{q}{\fraku} = \stars{q}{\frakv}$ with $q\in \frakM^{\star}(Z)$ and $\fraku, \frakv \in  \frakM(Z)$, then $\fraku = \frakv$.
\mlabel{lem:mstars}
\end{lemma}

\begin{proof}
We prove the result by induction on the order of $\stars{q}{\fraku}\geq \fraku$.
For the initial step, we have $\stars{q}{\fraku} = \fraku$. So $q = \star$ and $\fraku = \stars{q}{\fraku} = \stars{q}{\frakv} = \frakv$.
For the induction step, depending on the first symbol occurring in $q$ is a variable in $Z$, or a $\star$, or a bracket, we have the following cases to consider.

\noindent{\bf Case 1.} $q = xp$ for some $x\in Z$ and $p\in \frakM^{\star}(Z)$.
Then
$$x\stars{p}{\fraku} = \stars{q}{\fraku} = \stars{q}{\frakv} = x \stars{p}{\frakv},$$
and so $\stars{p}{\fraku} = \stars{p}{\frakv}$.
Since $\leq$ is a monomial order, we have $\stars{q}{\fraku} > \stars{p}{\fraku}$. By the induction hypothesis and $\stars{p}{\fraku} = \stars{p}{\frakv}$, we have $\fraku = \frakv$.

\noindent{\bf Case 2.} $q = \star w$ and $w\in \frakM(Z)$. Then $\fraku w =\stars{q}{\fraku} = \stars{q}{\frakv} = \frakv w$ and so $\fraku = \frakv$.

\noindent{\bf Case 3.} The first symbol in $q$ is a bracket. In this case, we have two subcases.

\noindent{\bf Case 3.1.} $q = \lc p\rc w$ for some  $p\in \frakM^\star(Z)$ and $w\in \frakM(Z)$. Then
$$\lc \stars{p}{\fraku} \rc w  = \stars{q}{\fraku} = \stars{q}{\frakv} = \lc \stars{p}{\frakv} \rc w $$
and so $\stars{p}{\fraku} = \stars{p}{\frakv}$. Since $\leq$ is a monomial order, we have $\stars{q}{\fraku}  > \stars{p}{\fraku} $. By the induction hypothesis and $\stars{p}{\fraku} = \stars{p}{\frakv}$, we get $\fraku = \frakv$.

\noindent{\bf Case 3.2.} $q = \lc w\rc p$ for some $w\in \frakM(Z)$ and $p\in \frakM^{\star}(Z)$. Then
$$\lc w\rc \stars{p}{\fraku} = \stars{q}{\fraku} = \stars{q}{\frakv} = \lc w\rc \stars{p}{\frakv}.$$
Thus $\stars{p}{\fraku} = \stars{p}{\frakv} $.
Again since $\leq$ is a monomial order, we get $\stars{q}{\fraku}  > \stars{p}{\fraku} $. By the induction hypothesis and $\stars{p}{\fraku} = \stars{p}{\frakv}$, we obtain $\fraku = \frakv$.
This completes the proof.
\end{proof}

\begin{lemma}
Let $Z$ be a set and $\leq$ a monomial order on $\frakM(Z)$. The $\Pi_{S_\Phi(Z)}$ from $\leq$ is a simple \term-rewriting system on $\bfk\mapm{Z}$.
\label{lemma:1nonrec}
\end{lemma}

\begin{proof}
We only need to show that $q|_{\bar{\phi(\fraku)}} \dps q|_{\re{\phi(\fraku)}}$ for any
$q\in \frakM^\star(Z)$ and $\phi(\fraku)\in S_\Phi(Z)$. If $\re{\phi(\fraku)} = 0$, there is nothing to prove. Suppose $\re{\phi(\fraku)} \neq 0$ and write
$$\re{\phi(\fraku)} = \sum_{i=1}^m c_i \sumre{\phi(\fraku)}{i},$$
where $c_i\in \nbfk$ and $\sumre{\phi(\fraku)}{i}$, $1\leq i\leq m$, are mutually distinct monomials of $\re{\phi(\fraku)}$. If  $q|_{\bar{\phi(\fraku)}} \dps q|_{\re{\phi(\fraku)}}$ fails, from Definition~\mref{def:dps}, there is some $1\leq i\leq m$ such that $\stars{q}{\bar{\phi(\fraku)}} = \stars{q}{\sumre{\phi(\fraku)}{i}}$. So $\bar{\phi(\fraku)} = \sumre{\phi(\fraku)}{i}$ by Lemma~\mref{lem:mstars}, contradicting $\bar{\phi(\fraku)} > \sumre{\phi(\fraku)}{i}$.
\end{proof}

The following result gives a sufficient condition for terminating.

\begin{lemma}{\rm (\mcite{GGSZ})}
Let $Z$ be a set, let $\leq$ be a monomial order on $\frakM(Z)$, and let $S \subseteq \bfk\frakM(Z)$ be monic. Then $\Pi_S$ from $\leq$ is terminating.
\mlabel{lem:terminating}
\end{lemma}

\subsection{\Gs OPIs and convergent OPIs}
In this subsection, we study the relationship between a \gs system of OPIs and a convergent system of OPIs.
In terms of $\star$-bracketed words, the operated ideals in $\bfk\mapm{Z}$ can be characterized~\cite{BCQ,GSZ} as follows.

\begin{lemma}{\rm (\cite[Lemma 3.2]{GSZ})}
Let $Z$ be a set and $S \subseteq \bfk\mapm{Z}$. Then
\begin{equation}\hspace{10pt}\Id(S) = \left\{\, \sum_{i=1}^n c_i q_i|_{s_i} \medmid n\geq 1
{\rm\ and\ } c_i\in \nbfk,
q_i\in \frakM^{\star}(Z), s_i\in S {\rm\ for\ } 1\leq i\leq n\,\right\}.
\label{eq:repgen}
\end{equation}
\mlabel{lem:opideal}
\end{lemma}

\begin{lemma}
Let $Z$ be a set, and let $\leq$ be a linear order on $\frakM(Z)$. Let $S \subseteq \bfk\frakM(Z)$ be monicized with respect to $\leq$, and let $\Pi_S$ be the \term-rewriting system from $\leq$. If $f\astarrow_{\pis} g$ for $f,g\in \bfk \frakM(Z)$, then $f-g\in \Id(S)$.
\mlabel{lem:inideal}
\end{lemma}

\begin{proof}
If $f =g$, then $f-g = 0\in \Id(S)$. Suppose $f\neq g$.
Let $n\geq 1$ be the minimum number such that $f$ rewrites to $g$ by $n$ steps. We prove the result by induction on $n$. If $n=1$, then $f\to_{\pis} g$. Write
$$f = c q|_{\bar{s}} \dps f_1 \to_{\Pi_S}  c q|_{R(s)} + f_1 = g,$$
where $c\in \nbfk$, $s\in S$ and $f_1\in \bfk \frakM(Z)$.
Then
$$f-g = c q|_{\bar{s}} - c q|_{R(s)} = cq|_{\bar{s} - R(s)} = cq|_{s} \in \Id(S).$$
Assume that the result is true for $n=m\geq 1$ and consider the case of $n=m+1\geq 2$. Then we have $f\to_{\pis} h \astarrow_{\pis} g$ for some $f\neq h\in \bfk\frakM(Z)$. By the induction hypothesis, $f-h\in \Id(S)$ and $h-g\in \Id(S)$. Thus
$f-g\in \Id(S)$, as required.
\end{proof}

\begin{lemma}
Let $Z$ be a set, and let $\leq$ be a linear order on $\frakM(Z)$. Let $S \subseteq \bfk\frakM(Z)$ be monicized with respect to $\leq$, and let $\Pi_S$ be the \term-rewriting system from $\leq$.
\begin{enumerate}
\item If $\,\pis$ is confluent, then $u\in \Id(S)$ if and only if $u\astarrow_{\pis} 0$. \mlabel{it:idealto0}

\item  If $\,\pis$ is confluent, then $\Id(S)\cap \bfk \Irr(S) =0$. \mlabel{it:confcap}

\item If $\,\pis$ is terminating and $\Id(S)\cap \bfk \Irr(S) =0$, then $\pis$ is confluent. \mlabel{it:capconf}

\item If $\,\pis$ is terminating, then $\bfk\frakM(Z) = \Id(S) +  \bfk \Irr(S)$, where $\irr(S)= \frakM(Z) \setminus \{ q|_{\lbar{s}} \mid s\in S \}$.
\mlabel{it:tsum}
\end{enumerate}
\mlabel{lem:useful}
\end{lemma}

\begin{proof}
Note that $\bfk\irr(S)$ is precisely the set of normal forms for $\Pi_S$.

(\mref{it:idealto0}) If $u\astarrow_{\pis} 0$, then $u\in \Id(S)$ from Lemma~\mref{lem:inideal}.
Conversely, let $u\in \Id(S)$. By Eq.~(\mref{eq:repgen}), we have
$$ u = \sum_{i=1}^k c_i q_i|_{s_i}, \, \text{ where }\, c_i\in \nbfk, s_i\in S, q_i\in \frakM^\star(Z), 1\leq i\leq k.$$
For each $s_i = \lbar{s_i} \dps (-\re{s_i})$ with $1\leq i\leq k$, we have
$$c_iq_i|_{s_i} = c_i \stars{q_i}{\lbar{s_i}} \dps ( -c_i \stars{q_i}{\re{s_i}}) \to_{\pis} c_i \stars{q_i}{\re{s_i}} - c_i \stars{q_i}{\re{s_i}} =0 \, \text{ and so }\, c_iq_i|_{s_i}\downarrow_{\pis} 0.$$
Since $\pis$ is confluent,
by Lemma~\mref{lem:red}(\mref{it:sumzero}), we have $u = \sum\limits_{i=1}^k c_i q_i|_{s_i} \astarrow_{\pis} 0$.

(\mref{it:confcap}) Suppose $\Id(S)\cap \bfk \Irr(S)\neq 0$. Let $0\neq w\in \Id(S)\cap \bfk \Irr(S)$. Since $w\in \bfk \Irr(S)$, $w$ is in normal form. On the other hand, from $w\in \Id(S)$ and Item~(\mref{it:idealto0}), we have $w\astarrow_{\pis} 0 $. So $w$ has two normal forms $w$ and $0$, contradicting that $\pis$ is confluent.

(\mref{it:capconf}) Suppose to the contrary that $\pis$ is not confluent. Since $\pis$ is terminating, there is $w\in \bfk\frakM(Z)$ such that $w$ has two distinct normal forms, say $u$ and $v$. Thus $u,v\in \bfk\Irr(S)$ and so $u-v\in \bfk\Irr(S)$. From Lemma~(\mref{lem:inideal}), $w-u\in \Id(S)$ and $w-v\in \Id(S)$.
Hence $0\neq u-v\in\Id(S) \cap \bfk\Irr(S)$, a contradiction.

(\mref{it:tsum}) Let $w\in \bfk\frakM(Z)$. Since $\pis$ is terminating, there is $u\in \bfk \Irr(S)$ such that $w\astarrow_\Pi u$. From Lemma~\mref{lem:inideal}, we have $w-u\in \Id(S)$ and so $w\in \Id(S) + \bfk \Irr(S)$.
\end{proof}

\begin{theorem}
Let $Z$ be a set, and let $\leq$ be a monomial order on $\frakM(Z)$. Let $S \subseteq \bfk\frakM(Z)$ be monicized with respect to $\leq$, and let $\Pi_S$ be the \term-rewriting system from $\leq$. Then the following statements are equivalent.
\begin{enumerate}
\item $\pis$ is convergent. \mlabel{it:sconv}

\item $\pis$ is confluent. \mlabel{it:sconf}

\item $\Id(S) \cap \bfk \Irr(S) = 0$. \mlabel{it:scap}

\item $\Id(S) \oplus \bfk \Irr(S) =\bfk \frakM(Z)$. \mlabel{it:ssum}

\item $S$ is a Gr\"{o}bner-Shirshov basis in $\bfk\frakM(Z)$ with respect to $\leq$. \mlabel{it:sgsbasis}

\end{enumerate}
\mlabel{thm:convsum}
\end{theorem}

\begin{proof}
Since $\leq $ is a monomial order on $\frakM(Z)$, $\pis$ is terminating by Lemma~\mref{lem:terminating}. So Item~(\mref{it:sconv}) and Item~(\mref{it:sconf}) are equivalent.
The equivalence of Item~(\mref{it:sconf}) and Item~(\mref{it:scap})
follows from Items~(\mref{it:confcap}) and~(\mref{it:capconf}) in Lemma~\mref{lem:useful}.

Clearly, Item~(\mref{it:ssum}) implies Item~(\mref{it:scap}). The converse employs Item~(\mref{it:tsum}) in Lemma~~\mref{lem:useful}. Finally, the equivalence of Item~(\mref{it:ssum}) and Item~(\mref{it:sgsbasis}) is obtained by Theorem~\mref{thm:cdl}.
\end{proof}

Now we are ready to give the relationship between the reformulations of Rota's Classification Problem.

\begin{theorem}
Let $\Phi \subseteq \bfk\frakM(X)$ be a system of OPIs.
\begin{enumerate}
\item For any set $Z$ and any monomial order $\leq$ on $\frakM(Z)$, $\Phi$ is \gs on $Z$ with respect to $\leq$ if and only if $\Phi$ is convergent on $Z$ with respect to the orientation $\lbar{S}_\Phi(Z) := \{\lbar{\phi(\fraku)} \mid \phi(\fraku)\in S_\Phi(Z) \}$ from $\leq$.
\mlabel{it:loiff}

\item
If $\Phi$ is \gs, then $\Phi$ is convergent.
\mlabel{it:geiff}

\item If $\Phi$ is \pgs, then $\Phi$ is potentially convergent.
\mlabel{it:geif}
\end{enumerate}
\mlabel{thm:sgoodtr}
\end{theorem}

\begin{proof}
(\mref{it:loiff}) Item~(\mref{it:loiff}) follows from applying Theorem~\mref{thm:convsum} to $S = S_\Phi(Z)$.

(\mref{it:geiff}) Suppose that $\Phi$ is \gs. By Definition~\mref{def:sggood}, for any set $Z$, there is a monomial order $\leq$ on $\frakM(Z)$ such that
$\Phi$ is \gs on $Z$ with respect to $\leq$. By Item~(\mref{it:loiff}), $\Phi$ is convergent on $Z$ with respect to the orientation $\lbar{S}_\Phi(Z)$ from $\leq$ and so is convergent.

(\mref{it:geif}) Suppose $\Phi$ is \pgs. From Definition~\mref{def:sggood}, for any set $Z$, there
is a monomial order on $\frakM(Z)$ such that $\Phi$ is \pgs on $Z$ with respect to $\leq$. By Definition~\mref{def:pgord}, there is a superset $\Phi'\subseteq \bfk\frakM(X)$ of $\Phi$ such that $\Id(S_\Phi(Z)) = \Id(S_{\Phi'}(Z))$ and $\Phi'$ is \gs on $Z$ with respect to $\leq$.
In view of Item~(\mref{it:loiff}), $\Phi'$ is convergent on $Z$ with respect to the the orientation $\lbar{S}_{\Phi'}(Z)$ from $\leq$. Hence $\Phi$ is potentially convergent.
\end{proof}

\begin{coro}
Let $\Phi$ be the system of (two-sided) averaging OPIs defined in Eq.~(\mref{eq:phi12}).
Then $\Phi$ is not \gs.
\mlabel{coro:avera}
\end{coro}

\begin{proof}
By Proposition~\mref{pp:anotsg}, $\Phi$ is not convergent. From Theorem~\mref{thm:sgoodtr}~(\mref{it:geiff}), $\Phi$ is not \gs.
\end{proof}

\begin{remark}
By~\cite[Theorems~2.41, 3.10]{GZ}, the system of averaging OPIs $\Phi$ in Corollary~\mref{coro:avera} can be extended to a set of OPIs that is Gr\"{o}bner-Shirshov. Thus $\Phi$ is \pgs and hence is \pcs.
\mlabel{re:apconv}
\end{remark}

\section{A sufficient condition for \gs  OPIs} \mlabel{sec:aclass}
In this section, we provide a sufficient condition for an OPI to be \gs. In Section~\ref{ss:thmexam} we give the statement of the theorem and show that previously known examples of \gs OPIs can be easily verified by this theorem. As another application, we prove that the modified Rota-Baxter OPI is \gs. The proof of the theorem is given in Section~\ref{ss:proof}.

\subsection{Statement of the main theorem and examples}
\label{ss:thmexam}

Like the differential operator and Rota-Baxter operator, many operators are defined by a single OPI. In this subsection, we consider a single OPI $\phi$ and supply a method to prove that $\phi$ is \gs.

Let $\phi = \phi(x_1, \ldots, x_k) = \lbar{\phi} - R(\phi) \in \bfk\mapm{X}$ be an OPI. In the rest of this paper, we write $\phi(\frakx)$ for $\phi(x_1, \ldots, x_k)$ in short.
We call $\phi(\frakx)$ {\bf multiple linear} (or {\bf totally linear}) if $\phi(\frakx)$ is linear in each variable $x_i$, $1\leq i\leq k$.
Let $Z$ be a set. We say that an element $f\in \bfk\frakM(Z)$ is in {\bf $\phi$-normal form} if no monomial of $f$ contains any subword of the form $\bar{\phi(\fraku)}$ with $\fraku\in \frakM(Z)^k$.

\begin{theorem}
Let $\phi(\frakx)\in \bfk\frakM(X)$ be a multi-linear OPI such that $\re{\phi(\frakx)}$ is in $\phi$-normal form.
Suppose that, for any set $Z$, there is a monomial order $\leq$ on $\frakM(Z)$, such that the following two conditions hold:
\begin{enumerate}
\setlength{\itemsep}{3pt}

\item
if $\phi(\fraku), \phi(\frakv)\in S_\phi(Z)$ are such that
$\bar{\phi(\fraku)} = ab$ and $\bar{\phi(\frakv)} = bc$ for some $a,b,c\in \frakM(Z)$ and $\fraku,\frakv\in \frakM(Z)^k$, then $\re{\phi(\fraku)}c \downarrow_{\phi} a\re{\phi(\frakv)}$,
where $\Pi_\phi:= \Pi_{S_\phi(Z)}$ is the \term-rewriting system from $\leq$.
\mlabel{it:rb3}

\item
if $\lbar{\phi(\fraku)} = q|_{\lbar{\phi(\frakv)}}$ for some $\star\neq q\in \frakM^\star(Z)$ and $\fraku, \frakv \in \frakM(Z)^k$, then $\bar{\phi(\frakv)}$ is a subword of some $u_i$, $1\leq i\leq k$.
\mlabel{it:incl}
\end{enumerate}
Then $\phi(\frakx)$ is \gs, as is its defined operator.
\mlabel{thm:sgoodtype}
\end{theorem}

We postpone the proof of Theorem~\ref{thm:sgoodtype} to Section~\ref{ss:proof} and first give some remarks and examples. 

\begin{remark}
Condition~(\mref{it:rb3}) is a necessary condition for $\phi(\frakx)$ to be a \gs  OPI. Indeed, let $\phi(\fraku), \phi(\frakv)\in S_\phi(Z)$ with
$\bar{\phi(\fraku)} = ab$ and $\bar{\phi(\frakv)} = bc$ for some $a,b,c\in \frakM(Z)$. Since $\phi(\frakx)$ is \gs , $S_\phi(Z)$ is a Gr\"{o}bner-Shirshov basis by Definition~\mref{def:sggood}. By Theorem~\mref{thm:convsum}, the \term-rewriting system $\Pi_\phi = \Pi_{S_\phi(Z)}$ from $\leq$ is confluent.
So for the local fork
$$(abc = \bar{\phi(\fraku)}c \to_{\phi} \re{\phi(\fraku)}c,\  abc = a\bar{\phi(\frakv)} \to_{\phi} a \re{\phi(\frakv)}),$$
we have $  \re{\phi(\fraku)}c \downarrow_\phi a \re{\phi(\frakv)}$.
\end{remark}

\begin{remark}
As a counter-example of condition~(\ref{it:incl}), consider $\phi(x) = \lc\lc x \rc\rc$, $q = \lc \star\rc$, $\fraku = \lc x\rc$ and $\frakv = x$. Then
$$\lbar{\phi(\fraku)} = \lc\lc \fraku\rc\rc = \lc\lc\lc x\rc\rc\rc =\stars{q}{\lc\lc x\rc\rc} = \stars{q}{\lbar{\phi(\frakv)}}.$$
But $\lbar{\phi(\frakv)} = \lc\lc x\rc\rc$ is not a subword of $\fraku = \lc x\rc$.

However, Item~(\mref{it:incl}) is not a necessary condition for $\phi(\frakx)$ to be a \gs  OPI. For example, let $\leq$ be a monomial order on $\frakM(Z)$ and $\phi(x) = \lc\lc x\rc\rc$. Then we have a \term-rewriting system from $\leq$
$$\Pi_{S_\phi(Z)} = \{\stars{q}{\lc\lc\fraku\rc\rc} \to 0 \mid q\in \frakM^\star(Z), \fraku\in \frakM(Z) \},$$
which is confluent. By Theorem~\mref{thm:convsum}, $\phi(x)$ is a \gs OPI. But as explained just above, $\phi(x)$ does not satisfy condition~(\mref{it:incl}).

\end{remark}

\begin{exam} {\bf (Differential type OPI)} \label{ex:difftype}
A differential type OPI~\mcite{GSZ}, defining a differential type operator, is
$$\phi(x_1,x_2) = \lc x_1 x_2\rc  - N(x_1,x_2),$$
where
\begin{enumerate}
\item $N(x_1, x_2)$ is multi-linear in $x_1$ and $x_2$;
\item $N(x_1, x_2)$ is in $\phi(x_1, x_2)$-normal form;
\item For any set $Z$ and $u,v,w \in \frakM(Z)\setminus\{1\}$,
\begin{equation}
N(uv, w) - N(u, vw) \astarrow_\phi 0.
\mlabel{eq:diff0}
\end{equation}

\end{enumerate}

We verify that,  with respect the monomial order $\leq$ defined in~\cite{GSZ}, $\phi(x_1,x_2)$ satisfies the conditions~(\mref{it:rb3}) and~(\mref{it:incl}) in Theorem~\mref{thm:sgoodtype} and therefore is a \gs OPI. This gives another proof of~\cite[Theorem~5.7]{GSZ}. We begin with verifying the first condition. Let $\Pi_\phi = \Pi_{S_\phi(Z)}$ be the \term-rewriting system from $\leq$. Note that
$$\lbar{\phi(u_1,u_2)} = \lc u_1 u_2\rc\, \text{ and }\, \re{\phi(u_1, u_2)} = N(u_1, u_2)\,\text{ for }\, u_1, u_2\in \frakM(Z).$$
Let $\phi(u_1, u_2)$ and $\phi(v_1, v_2)$ be in $S_\phi(Z)$ such that
$$\bar{\phi(u_1, u_2)} = ab \, \text{ and }\, \bar{\phi(v_1, v_2)} = bc \, \text{ for some }\, u_1, u_2, v_1, v_2\in \frakM(Z)\setminus\{1\}, a,b,c\in \frakM(Z).$$
Then
$$\bar{\phi(u_1, u_2)}=\lc u_1  u_2\rc = ab\, \text{ and }\, \bar{\phi(v_1, v_2)}=\lc v_1 v_2\rc = bc.$$
So
\begin{equation}
a=c=1\, \text{ and }\, b= \lc u_1u_2\rc = \lc v_1v_2\rc.
\mlabel{eq:diffuv}
\end{equation}
Note that
\begin{equation}
\begin{aligned}
\bar{\phi(u_1, u_2)} &= \lc u_1  u_2\rc \to_\phi N(u_1, u_2)=\re{\phi(u_1, u_2)} = \re{\phi(u_1, u_2)}c,\\
\bar{\phi(v_1, v_2)} &= \lc v_1 v_2\rc \to_\phi N(v_1, v_2)=\re{\phi(v_1, v_2)} = a\re{\phi(v_1, v_2)}.
\end{aligned}
\mlabel{eq:diffnuv}
\end{equation}
From Eq.~(\mref{eq:diffuv}), we have $u_1u_2 = v_1v_2$. If $u_1 = v_1$, then $u_2 = v_2$ and $\phi(u_1,u_2) = \phi(v_1, v_2)$. So
$$\re{\phi(u_1, u_2)} \downarrow_\phi \re{\phi(v_1, v_2)},\quad \re{\phi(u_1, u_2)}c \downarrow_\phi a\re{\phi(v_1, v_2)},$$
by the fact that $a=c=1$ in Eq.~(\mref{eq:diffuv}). Suppose $u_1\neq v_1$. Since $u_1u_2 = v_1v_2$, either $u_1 = v_1v$ or $v_1 = u_1v$ for some $v\in \frakM(Z)\setminus\{1\}$. In the former case, we have $u_1u_2 = v_1vu_2 = v_1v_2$ and so $ vu_2= v_2$. From Eqs.~(\mref{eq:diff0}) and~(\mref{eq:diffnuv}),
$$\re{\phi(u_1, u_2)}c - a\re{\phi(v_1, v_2)} = N(u_1, u_2) -  N(v_1, v_2) =N(v_1v, u_2) -  N(v_1, vu_2)   \astarrow_\phi 0.$$
Using Lemma~\mref{lem:red}(\mref{it:join}),
$$\re{\phi(u_1, u_2)}c \downarrow_\phi a\re{\phi(v_1, v_2)}.$$
In the latter case of $v_1 = u_1v$, we get $u_2=vv_2$ and
$$a\re{\phi(v_1, v_2)} - \re{\phi(u_1, u_2)}c  = N(v_1, v_2) - N(u_1, u_2) =N(u_1v, v_2) -  N(u_1, vv_2)   \astarrow_\phi 0.$$
So $$a\re{\phi(v_1, v_2)} \downarrow_\phi \re{\phi(u_1, u_2)}c.$$
To verify condition~(\mref{it:incl}) in Theorem~\ref{thm:sgoodtype}, let
$$\lc u_1 u_2\rc = \bar{\phi(u_1, u_2)} = q|_{\bar{\phi(v_1, v_2)}} = q|_{\lc v_1 v_2\rc}$$
for some $\star\neq q\in \frakM^\star(Z)$ and $u_1, u_2, v_1,v_2\in \frakM(Z)\setminus\{1\}.$
Since $q\neq \star$, $\lc u_1u_2\rc \neq \lc v_1 v_2\rc$ and so $\lc v_1  v_2\rc$ is a subword of $u_1u_2$. Since the breadth of $\lc v_1v_2\rc$ is 1, $\lc v_1  v_2\rc$
is a subword of $u_1$ or $u_2$, as needed.
\label{ex:diff}
\end{exam}

\begin{exam} {\bf (Rota-Baxter type OPI)} 
A Rota-Baxter type OPI \mcite{GGSZ}, defining a Rota-Baxter type operator, is
$$\phi(x_1,x_2) = \lc x_1\rc\lc x_2\rc - \lc B(x_1,x_2)\rc,$$
where $B(x_1,x_2)$ satisfies
\begin{enumerate}
\item $B(x_1, x_2)$ is multi-linear in $x_1$ and $x_2$;
\item $B(x_1, x_2)$ is in $\phi(x_1, x_2)$-normal form;
\item The \term-rewriting system $\Pi_{S_\phi(Z)}$ is terminating;
\item For any set $Z$ and $u,v,w \in \frakM(Z)$,
\begin{equation}
B(B(u,v), w) - B(u, B(v,w)) \astarrow_\phi 0.
\mlabel{eq:rbuvw}
\end{equation}
\end{enumerate}
We show that $\phi(x_1,x_2)$ satisfies the two conditions in Theorem~\mref{thm:sgoodtype} with respect the monomial order $\leq_\db$ defined in~\cite{GGSZ} and therefore is a \gs OPI. This gives another proof of~\cite[Theorem~4.9]{GGSZ}. Let $\Pi_\phi = \Pi_{S_\phi(Z)}$ be the \term-rewriting system from $\leq_\db$. To verify condition~(\mref{it:rb3}) in Theorem~\mref{thm:sgoodtype},
note that
$$\lbar{\phi(u_1, u_2)} = \lc u_1\rc\lc u_2\rc\,\text{ and }\, \re{\phi(u_1, u_2)} = \lc B(u_1,u_2)\rc\,\text{ for }\, u_1, u_2\in \frakM(Z).$$
Let $\phi(u_1, u_2)$ and $\phi(v_1, v_2)$ be in $S_\phi(Z)$ such that
$$\bar{\phi(u_1, u_2)} = ab \, \text{ and }\, \bar{\phi(v_1, v_2)} = bc\, \text{ for some } u_1, u_2, v_1, v_2, a,b,c\in \frakM(Z).$$
Then
$$\bar{\phi(u_1, u_2)} = \lc u_1\rc \lc u_2\rc = ab \, \text{ and }\,  \bar{\phi(v_1, v_2)} = \lc v_1\rc \lc v_2\rc = bc.$$
Thus
\begin{equation}
a =\lc u_1\rc, b = \lc u_2\rc = \lc v_1\rc, c = \lc v_2\rc\, \text{ and }\, u_2 = v_1.
\mlabel{eq:uv}
\end{equation}
So
\begin{align*}
\re{\phi(u_1, u_2)}c &= \lc B(u_1, u_2)\rc\lc v_2\rc \to_{\phi} \lc B(B(u_1, u_2), v_2)\rc,\\
a \re{\phi(v_1, v_2)} &= \lc u_1\rc \lc B(u_2, v_2)\rc \to_{\phi} \lc B(u_1, B(u_2, v_2))\rc.
\end{align*}
It follows from Eq.~(\mref{eq:rbuvw}) that
\begin{align*}
&\re{\phi(u_1, u_2)}c - a \re{\phi(v_1, v_2)} = \lc B(B(u_1, u_2), v_2)-  B(u_1, B(u_2, v_2))\rc \astarrow_\phi 0.
\end{align*}
By Lemma~\mref{lem:red}(\mref{it:join}), we have
$$\re{\phi(u_1, u_2)}c \downarrow_\phi a \re{\phi(v_1, v_2)}.
$$
Hence condition~(\mref{it:rb3}) in Theorem~\mref{thm:sgoodtype} holds. For condition~(\mref{it:incl}) in Theorem~\mref{thm:sgoodtype}, let $$\lc u_1\rc \lc u_2\rc = \bar{\phi(u_1, u_2)} = q|_{\bar{\phi(v_1, v_2)}} = q|_{\lc v_1\rc \lc v_2\rc}$$
for some $\star\neq q\in \frakM^\star(Z)$ and $u_1, u_2, v_1,v_2\in \frakM(Z)$. Since $q\neq \star$, $\lc u_1\rc \lc u_2\rc \neq \lc v_1\rc \lc v_2\rc$ and so
$\lc v_1\rc \lc v_2\rc$ is a subword of $\lc u_1\rc$ or $\lc u_2\rc$. Since the breadth of $\lc u_1\rc$ is 1 and the breadth of $\lc v_1\rc \lc v_2\rc$ is 2, $\lc u_1\rc \neq \lc v_1\rc \lc v_2\rc$. Similarly, $\lc u_2\rc \neq \lc v_1\rc \lc v_2\rc$. Hence $\lc v_1\rc \lc v_2\rc$ is a subword of $u_1$ or $u_2$,
as required.
\mlabel{ex:rbtype}
\end{exam}

We finally give an application to an OPI that has been been considered in the context of Rota's Classification Problem before.
The {\bf modified Rota-Baxter OPI} of weight $\lambda$ is
$$\phi(x_1,x_2) = \lc x_1\rc \lc x_2\rc - \lc x_1\lc x_2\rc\rc - \lc \lc x_1\rc x_2\rc -\lambda x_1x_2,\,\text{ where } \lambda \in \bfk. $$
When $\lambda=-\mu^2$, this gives~\mcite{Fard}
$$ P(x_1)P(x_2)=P(x_1P(x_2))+P(P(x_1)x_2) - \mu^2 x_1x_2,$$
as an associative analog of the modified classical Yang-Baxter equation on Lie algebras~\cite{Sha}.
Note the subtle difference between this operator and the Rota-Baxter operator.

\begin{theorem}
The modified Rota-Baxter OPI is \gs.
\mlabel{thm:mrb}
\end{theorem}

\begin{proof}
For the proof, we verify that the OPI satisfies the conditions in  Theorem~\ref{thm:sgoodtype} for the monomial order $\leq_\db$ defined in \mcite{GGSZ}. Let $\Pi_\phi=\Pi_{S_\phi(Z)}$ be the
\term-rewriting system from $\leq_\db$. With the order, we have
$$\bar{\phi(u_1, u_2)} = \lc u_1\rc \lc u_2\rc\, \text{ and }\,  \re{\phi(u_1,u_2)} = \lc u_1\lc u_2\rc\rc + \lc \lc u_1\rc u_2\rc + \lambda u_1u_2\,\text{ for }\, u_1, u_2\in \frakM(Z).$$
Since $\phi(u_1, u_2)$ has the same leading monomial as the one for Rota-Baxter type operators, by the same argument as for Example~\mref{ex:rbtype}, condition~(\mref{it:incl}) in Theorem~\mref{thm:sgoodtype} holds. Now we show that condition~(\mref{it:rb3}) is also fulfilled.
With notations in Example~\mref{ex:rbtype} and from Eq.~(\mref{eq:uv}), we have
$$R(\phi(u_1, u_2))c = ( \lc u_{1}\lc u_{2}\rc\rc  + \lc \lc u_{1}\rc u_{2}\rc +\lambda u_{1}u_{2}) \lc v_{2}\rc $$
and
$$ a R(\phi(v_1, v_2)) =\lc u_{1}\rc R(\phi(u_2, v_2)) = \lc u_{1}\rc (\lc u_{2}\lc v_{2}\rc\rc + \lc \lc u_{2}\rc v_{2}\rc +\lambda u_{2}v_{2}).$$
On the one hand, we have
\begin{align*}
&R(\phi(u_1, u_2))c \\
=&\ \lc u_{1}\lc u_{2}\rc\rc \lc v_{2}\rc + \lc \lc u_{1}\rc u_{2}\rc \lc v_{2}\rc +\lambda u_{1}u_{2}\lc v_{2}\rc = \lc u_{1}\lc u_{2}\rc\rc \lc v_{2}\rc \dps (\lc \lc u_{1}\rc u_{2}\rc \lc v_{2}\rc +\lambda u_{1}u_{2}\lc v_{2}\rc)\\
\to_\phi&\ \lc u_{1}\lc u_{2}\rc \lc v_{2}\rc\rc + \lc \lc u_{1}\lc u_{2}\rc\rc v_{2}\rc +\lambda u_{1}\lc u_{2}\rc v_{2} + \lc \lc u_{1}\rc u_{2}\rc \lc v_{2}\rc +\lambda u_{1}u_{2}\lc v_{2}\rc\\
=&\ \lc \lc u_{1}\rc u_{2}\rc \lc v_{2}\rc \dps( \lc u_{1}\lc u_{2}\rc \lc v_{2}\rc\rc + \lc \lc u_{1}\lc u_{2}\rc\rc v_{2}\rc +\lambda u_{1}\lc u_{2}\rc v_{2}+\lambda u_{1}u_{2}\lc v_{2}\rc)\\
\to_\phi&\ \lc \lc u_{1}\rc u_{2}\rc \lc v_{2}\rc \rc + \lc \lc \lc u_{1}\rc u_{2}\rc v_{2}\rc +\lambda \lc u_{1}\rc u_{2}v_{2} + \lc u_{1}\lc u_{2}\rc \lc v_{2}\rc\rc  \\
&\ + \lc \lc u_{1}\lc u_{2}\rc\rc v_{2}\rc +\lambda u_{1}\lc u_{2}\rc v_{2} +\lambda u_{1}u_{2}\lc v_{2}\rc\\
=&\ \lc u_{1}\lc u_{2}\rc \lc v_{2}\rc\rc \dps (\lc \lc u_{1}\rc u_{2}\rc \lc v_{2}\rc \rc + \lc \lc \lc u_{1}\rc u_{2}\rc v_{2}\rc +\lambda \lc u_{1}\rc u_{2}v_{2}  \\
&\ + \lc \lc u_{1}\lc u_{2}\rc\rc v_{2}\rc +\lambda u_{1}\lc u_{2}\rc v_{2} +\lambda u_{1}u_{2}\lc v_{2}\rc)\\
\to_{\phi}&\ \lc u_{1}\lc u_{2}\lc v_{2}\rc\rc\rc + \lc u_{1}\lc \lc u_{2}\rc v_{2}\rc\rc +\lambda\lc u_{1}u_{2}v_{2}\rc +\lc \lc u_{1}\rc u_{2}\rc \lc v_{2}\rc \rc + \lc \lc \lc u_{1}\rc u_{2}\rc v_{2}\rc   \\
&\ +\lambda \lc u_{1}\rc u_{2}v_{2} + \lc \lc u_{1}\lc u_{2}\rc\rc v_{2}\rc +\lambda u_{1}\lc u_{2}\rc v_{2} +\lambda u_{1}u_{2}\lc v_{2}\rc.
\end{align*}
On the other hand, we have
\begin{align*}
&a R(\phi(v_1, v_2)) \\
=&\ \lc u_{1}\rc \lc u_{2}\lc v_{2}\rc\rc + \lc u_{1}\rc \lc \lc u_{2}\rc v_{2}\rc +\lambda \lc u_{1}\rc u_{2}v_{2} = \lc u_{1}\rc \lc u_{2}\lc v_{2}\rc\rc \dps (\lc u_{1}\rc \lc \lc u_{2}\rc v_{2}\rc +\lambda \lc u_{1}\rc u_{2}v_{2})\\
\to_\phi&\ \lc u_{1}\lc u_{2}\lc v_{2}\rc\rc\rc + \lc \lc u_{1}\rc u_{2}\lc v_{2}\rc\rc +\lambda u_{1}u_{2}\lc v_{2}\rc + \lc u_{1}\rc \lc \lc u_{2}\rc v_{2}\rc +\lambda \lc u_{1}\rc u_{2}v_{2}\\
=&\ \lc u_{1}\rc \lc \lc u_{2}\rc v_{2}\rc \dps (\lc u_{1}\lc u_{2}\lc v_{2}\rc\rc\rc + \lc \lc u_{1}\rc u_{2}\lc v_{2}\rc\rc +\lambda u_{1}u_{2}\lc v_{2}\rc +\lambda \lc u_{1}\rc u_{2}v_{2})\\
\to_\phi&\ \lc u_{1}\lc\lc u_{2}\rc v_{2}\rc \rc + \lc \lc u_{1}\rc \lc u_{2}\rc v_{2}\rc +\lambda u_{1}\lc u_{2}\rc v_{2} + \lc u_{1}\lc u_{2}\lc v_{2}\rc\rc\rc \\
&\ + \lc \lc u_{1}\rc u_{2}\lc v_{2}\rc\rc +\lambda u_{1}u_{2}\lc v_{2}\rc +\lambda \lc u_{1}\rc u_{2}v_{2}\\
=&\ \lc \lc u_{1}\rc \lc u_{2}\rc v_{2}\rc \dps (\lc u_{1}\lc\lc u_{2}\rc v_{2}\rc \rc  +\lambda u_{1}\lc u_{2}\rc v_{2} + \lc u_{1}\lc u_{2}\lc v_{2}\rc\rc\rc \\
&\ + \lc \lc u_{1}\rc u_{2}\lc v_{2}\rc\rc +\lambda u_{1}u_{2}\lc v_{2}\rc +\lambda \lc u_{1}\rc u_{2}v_{2})\\
\to_\phi&\ \lc \lc u_{1}\lc u_{2}\rc \rc v_{2}\rc + \lc \lc \lc u_{1}\rc  u_{2}\rc v_{2}\rc
+\lambda \lc u_{1}u_{2}v_{2}\rc + \lc u_{1}\lc\lc u_{2}\rc v_{2}\rc \rc  +\lambda u_{1}\lc u_{2}\rc v_{2} \\
&\  + \lc u_{1}\lc u_{2}\lc v_{2}\rc\rc\rc + \lc \lc u_{1}\rc u_{2}\lc v_{2}\rc\rc +\lambda u_{1}u_{2}\lc v_{2}\rc +\lambda \lc u_{1}\rc u_{2}v_{2}.
\end{align*}
Hence
$$\re{\phi(a,b)}c \downarrow_\phi a\re{\phi(b,c)}$$
and so condition~(\mref{it:rb3}) is verified. This completes the proof.
\end{proof}

As a consequence, we obtain a construction of free modified Rota-Baxter algebras.
For a set $Z$, denote
\begin{equation*}
\rbw(Z): = \frakM(Z) \setminus \{ q|_{\lc u\rc \lc v\rc} \mid u,v\in \frakM(Z) \} =\frakM(Z) \setminus \{ q|_{\lbar{s}} \mid s\in S_\phi(Z) \}=:\Irr(S_\phi(Z)),
\end{equation*}
where $S_\phi(Z)$ is defined in Eq.~(\mref{eq:genphi}).

\begin{coro}
Let $Z$ be a set.
We have the following module isomorphism
$$ \bfk\frakM(Z)/\Id(S_\phi(Z)) \cong \bfk \rbw(Z).$$
More precisely,
$$ \bfk\frakM(Z)  = \Id(S_\phi(Z)) \oplus \bfk \rbw(Z)  .$$
\mlabel{thm:pbwbasis}
\end{coro}

\begin{proof}
This follows from Theorems~\ref{thm:cdl} and~\ref{thm:mrb}.
\end{proof}

\subsection{The proof of Theorem~\ref{thm:sgoodtype}}
\label{ss:proof}

Before starting the proof of Theorem~\ref{thm:sgoodtype}, we recall the following concepts~\cite{ZheG}.

\begin{defn} Let $Z$ be a set.
The particular location of the subword $u$ in the word
$w$ under the substitution $q|_u$ is called the {\bf  placement of
$u$ in $w$ by $q$}, denoted by $(u,q)$ for distinction.
\end{defn}

 A subword $u$ may appear at multiple locations (and hence have distinct placements using distinct $q$'s) in a bracketed word $w$. For example, there are two placements of $x$ in $w = x\lc x\rc\in \frakM(x)$, given by $(x,q_1)$ and $(x,q_2)$ where $q_1=\star\lc x\rc$ and $q_2=x\lc \star\rc$.

\begin{defn}
Let $Z$ be a set and $w\in \mapm{Z}$ such that
\begin{equation*}
\stars{q_1}{u_1}=w=\stars{q_2}{u_2} \, \text{ for some }\,u_1, u_2\in \mapm{Z}, q_1, q_2\in \frakM^\star(Z).
\end{equation*}
\noindent
The two \plas $(u_1,q_1)$ and $(u_2,q_2)$ are called
\begin{enumerate}
\item
{\bf  separated} if there exist $p \in \frakM^{\star_1,\star_2}(Z)$ and $a,b \in \mapm{Z}$ such that $\stars{q_1}{\star_1}=p|_{\star_1,\,b}$, $\stars{q_2}{\star_2}=p|_{a,\,\star_2}$, and $w=p|_{a,\,b}$;
\label{item:bsep}
\item
{\bf  nested} if there exists $q \in \frakM^{\star}(Z)$ such that either $q_2=\stars{q_1}q$ or $q_1=\stars{q_2}q$;
\label{item:bnes}
\item
 {\bf  intersecting} if there exist $q \in \frakM^{\star}(Z)$ and  $a, b, c \in \frakM(Z)\backslash\{1\}$ such that $w=q|_{abc}$ and either
\begin{enumerate}
\item
$ q_1=q|_{\star c}$ and $q_2=q|_{a\star}$; or
\item
$q_1=q|_{a\star}$ and $q_2=q|_{\star c}$.
\end{enumerate}
\label{item:bint}
\end{enumerate}
\label{defn:bwrel}
\end{defn}

\begin{prop}~{\rm (\cite[Theorem~4.11]{ZheG})}
Let $Z$ be a set and $w\in\frakM(Z)$. Any two \plas $(u_1, q_1)$ and $(u_2,q_2)$ in $w$ are either separated or nested or intersecting.
\label{pp:thrrel}
\end{prop}

Now we are ready for the proof of Theorem~\mref{thm:sgoodtype}.

\begin{proof}[Proof of Theorem~\mref{thm:sgoodtype}]
Let $\phi=\phi(\frakx)$ and
$\Pi_\phi$ the \term-rewriting system from $\leq$. We prove the result by showing that $\phi$ is \gs with respect to $\leq$.
By Theorem~\mref{thm:sgoodtr}~(\mref{it:loiff}), it suffices to prove that $\phi$ is convergent on $Z$ with respect to the orientation from $\leq$, that is, $\Pi_\phi$ is convergent by Definition~\mref{def:cwpp}~(\mref{it:pwrtp}).

Since $\leq$ is a monomial order on $\frakM(Z)$, $\Pi_\phi$ is terminating by Lemma~\mref{lem:terminating}. From Lemma~\mref{lem:newman}, we are left to show that $\Pi_\phi$ is locally confluent.
Since
$$\bar{R(\phi(\fraku))} < \bar{\phi(\fraku)}\,\text{ and }\, \stars{q}{\bar{\re{\phi(\fraku)}}}<  \stars{q}{\bar{\phi(\fraku)}}\,\text{ for } q\in \frakM^\star(Z), \phi(\fraku)\in S_\phi(Z),$$
$\Pi_\phi$ is compatible with $\leq$.
Using Lemma~\mref{lem:ltcon}, it suffices to show $\Pi_\phi$ is locally base-confluent, that is, for any local base-fork
$(dw\to_{\phi} dv_1, dw\to_{\phi} dv_2)$, we have $dv_1 - dv_2 \astarrow_\phi 0$. Suppose to the contrary that $\Pi_\phi$ is not locally base-confluent. Then $\frakC \neq \emptyset$, where
$$
\frakC = \left\{ w\in \frakM(Z) \left | \begin{array}{l}\text{ there is a local fork base-fork }(d w\to_{\phi} d v_1, d w\to_{\phi} dv_2) \\
\text{ for some $d\in \nbfk$ such that } dv_1 - dv_2 \not\astarrow_\phi 0\end{array} \right. \right \}.
$$
Since $\leq$ is a well-order, $\frakC$ has the least element with respect to $\leq$, say $w$.
Thus there are some $q_1, q_2\in \frakM^\star(Z)$, $\fraku, \frakv\in \frakM(Z)^k$ and $d\in \nbfk$ such that
\begin{equation}
\begin{aligned}
w &= \stars{q_1}{\bar{\phi(\fraku)}} =  \stars{q_2}{\bar{\phi(\frakv)}} \in \frakM(Z),\,
d w \to_{\phi} d\stars{q_1}{\re{\phi(\fraku)}},\\
dw &\to_{\phi} d\stars{q_2}{\re{\phi(\frakv)}},\, \text{ and }\, d\stars{q_1}{\re{\phi(\fraku)}} - d\stars{q_2}{\re{\phi(\frakv)}}\not\astarrow_\phi 0.
\mlabel{eq:contra}
\end{aligned}
\end{equation}
Let
$$Y:= \{u\in \frakM(Z) \mid u < w\}\,\text{ and }\, \Pi_{\bfk Y} =\Pi_\phi \cap (Y\times \bfk Y).$$
Since $\leq$ is a monomial order, we have
\begin{equation}
\bar{\stars{q_1}{\re{\phi(\fraku)}}} =  \stars{q_1}{\bar{\re{\phi(\fraku)}}}<  \stars{q_1}{\bar{\phi(\fraku)}} = w, \quad \, \bar{ \stars{q_2}{\re{\phi(\frakv)}}}  =  \stars{q_2}{\bar{\re{\phi(\frakv)}}}<  \stars{q_2}{\bar{\phi(\frakv)}} = w
\mlabel{eq:less}
\end{equation}
and
\begin{equation}
\stars{q_1}{\re{\phi(\fraku)}},\, \stars{q_2}{\re{\phi(\frakv)}}\in \bfk Y.
\mlabel{eq:less1}
\end{equation}
So $Y\neq \emptyset$. For any $f\astarrow_\phi g$ with $f\in \bfk Y$, since $\leq$ is compatible with $\Pi_\phi$, we get $\bar{g}\leq \bar{f}$ and so $g\in \bfk Y$. Thus $\Pi_{\bfk Y}$ is closed under $\Pi_\phi$. By Proposition~\mref{pp:subre}, we conclude that $\Pi_{\bfk Y} \leq \Pi_\phi$ is a sub-\term-rewriting system of $\Pi_\phi$. For any local base-fork $(ey\to_{\Pi_{\bfk Y}} ev_1, ey \to_{\Pi_{\bfk Y}} e v_2)$ of $\Pi_{{\bfk Y}}$ with $e\in \nbfk$, $y\in Y$ and $v_1, v_2\in \bfk Y$, it induces a local base-fork $(ey\to_{\phi} ev_1, ey \to_{\phi} e v_2)$ of $\Pi_\phi$. Since $y\in Y$, we have $y<w$ and $y\notin \frakC$ by the minimality of $w$. So $ev_1 - ev_2\astarrow_{\phi} 0$ by the definition of $\frakC$. Since $\Pi_{\bfk Y}\leq \Pi_\phi$ and $ev_1 - ev_2\in \bfk Y$, we have $ev_1 - ev_2\astarrow_{\Pi_{\bfk Y}} 0$. Thus $\Pi_{\bfk Y}$ is locally base-confluent and so is confluent by Lemmas~\mref{lem:newman} and~\mref{lem:ltcon}.

Since $\phi(\frakx)$ is multi-linear, we may write
$$\re{\phi(\frakx)} = \sum_{t=1}^m r_t\stars{p_t}{\frakx}:= \sum_{t=1}^m r_t \stars{p_t}{x_1, \ldots, x_k},$$
where $r_t\in \nbfk$, $p_t\in \frakM^{\star k} (Z)$ and $\stars{p_t}{x_1, \ldots, x_k}$, $1\leq t\leq m$,
are mutually distinct monomials.
Then
\begin{equation}
\begin{aligned}
\re{\phi(\fraku)} &= \sum_{t=1}^m r_t  \stars{p_t}{\fraku} := \sum_{t=1}^m r_t  \stars{p_t}{u_1, \ldots, u_k},\\
\re{\phi(\frakv)} &= \sum_{t=1}^m r_t  \stars{p_t}{\frakv} := \sum_{t=1}^m r_t  \stars{p_t}{v_1, \ldots, v_k},
\end{aligned}
\mlabel{eq:sumuv}
\end{equation}
and by Eq.~(\mref{eq:less1}),
\begin{equation}
\stars{q_1}{\stars{p_t}{\fraku}},\, \stars{q_1}{\stars{p_t}{\frakv}} \in \bfk Y \,\text{ for } 1\leq t\leq m.
\mlabel{eq:uvless}
\end{equation}
By Proposition~\ref{pp:thrrel},  these two \plas $(\bar{\phi(\fraku)},q_1)$ and $(\bar{\phi(\frakv}),q_2)$ in $w$ have three possible relative locations.

\smallskip

\noindent {\bf Case I: Separate placements}. By Definition~\ref{defn:bwrel}, there exists $p\in \frakM^{\star_1,\star_2}(Z)$ such that
$$\stars{q_1}{\star_1}=p|_{\star_1,\,\bar{\phi(\frakv)}}\,\text{ and }\,
\stars{q_2}{\star_2}=p|_{\bar{\phi(\fraku)},\,\star_2}.
$$
So
\begin{align}
\bfk Y \ni \stars{q_1}{\re{\phi(\fraku)}} =  p|_{\re{\phi(\fraku)},\,\bar{\phi(\frakv)}}
= \sum_{t=1}^m   r_t p|_{\stars{p_t}{\fraku},\,\bar{\phi(\frakv)}},
\mlabel{eq:seq0}
\end{align}
where the last step employs Eq.~(\mref{eq:sumuv}). For each $1\leq t\leq m$,
$$ p|_{\stars{p_t}{\fraku},\,\bar{\phi(\frakv)}} \to_{\Pi_{\bfk Y}} p|_{\stars{p_t}{\fraku},\,\re{\phi(\frakv)}}\,\text{ and so }\, p|_{\stars{p_t}{\fraku},\,\bar{\phi(\frakv)}} \downarrow_{\Pi_{\bfk Y}} p|_{\stars{p_t}{\fraku},\,\re{\phi(\frakv)}}.$$
Since $\Pi_{\bfk Y}$ is confluent,
\begin{align}
\stars{q_1}{\re{\phi(\fraku)}}  =  \sum_{t=1}^m  r_t p|_{\stars{p_t}{\fraku},\,\bar{\phi(\frakv)}} \downarrow_{\Pi_{\bfk Y}} \sum_{t=1}^m  r_t p|_{\stars{p_t}{\fraku},\,\re{\phi(\frakv)}} = \sum_{t,s=1}^m r_t r_s p|_{\stars{p_t}{\fraku},\,\stars{p_s}{\frakv}},
\mlabel{eq:seq1}
\end{align}
where the first equation follows from Eq.~(\mref{eq:seq0}), the confluence step from Lemmas~\mref{lem:cons} and~\mref{lem:red}(\mref{it:plus}), and the next equation from Eq.~(\mref{eq:sumuv}).
On the other hand,
\begin{align}
\stars{q_2}{\re{\phi(\frakv)}} = p|_{\bar{\phi(\fraku)},\, \re{\phi(\frakv)} }
=\sum_{s=1}^m  r_s p|_{\bar{\phi(\fraku)},\, \stars{p_s}{\frakv} }
\downarrow_{\Pi_{\bfk Y}} \sum_{s=1}^m  r_s p|_{\re{\phi(\fraku)},\, \stars{p_s}{\frakv} }
=\sum_{t,s=1}^m r_t r_s p|_{\stars{p_t}{\fraku},\,\stars{p_s}{\frakv}}.
\mlabel{eq:seq2}
\end{align}
Since $\Pi_{\bfk Y}$ is confluent, by Lemma~\mref{lem:red}(\mref{it:trans}), Eqs.~(\ref{eq:seq1}) and (\ref{eq:seq2}) we obtain
$$\stars{q_1}{\re{\phi(\fraku)}} \downarrow_{\Pi_{\bfk Y}} \stars{q_2}{\re{\phi(\frakv)}}.$$
Then it follows from Eq.~(\ref{eq:less1}) and Remark~\mref{re:tozero} that
$$  \stars{q_1}{\re{\phi(\fraku)}}-  \stars{q_2}{\re{\phi(\frakv)}}  \astarrow_{\Pi_{\bfk Y}} 0.
$$
By $\Pi_{\bfk Y} \leq \Pi_\phi$ being a sub-term-rewriting system and Lemma~\mref{lem:cons}, we have
$$\stars{q_1}{\re{\phi(\fraku)}}-  \stars{q_2}{\re{\phi(\frakv)}}  \astarrow_{\phi} 0\, \text{ and }\, d\stars{q_1}{\re{\phi(\fraku)}}-  d\stars{q_2}{\re{\phi(\frakv)}}  \astarrow_{\phi} 0,$$
contradicting Eq.~(\mref{eq:contra}).

\smallskip

\noindent {\bf Case II: Intersecting placements}.
By the symmetry of (i) and (ii) in Item~(\ref{item:bint}) of Definition~\ref{defn:bwrel}, we may assume that Item~(\ref{item:bint})\,(i) holds and hence $q_1 \ne q_2$. So there exist $q \in \frakM^{\star}(Z)$ and  $a, b, c \in \frakM(Z)\backslash\{1\}$ such that $w=q|_{abc}$,
$q_1=q|_{\star c}$ and $q_2=q|_{a\star}$. Then
$$\stars{q_1}{\re{\phi(\fraku)}} = q|_{\re{\phi(\fraku)} c}\, \text{ and }\, \stars{q_2}{\re{\phi(\frakv)}} = q|_{a \re{\phi(\frakv)} }.$$
So from Eq.~(\mref{eq:less}),
$$ \bar{q|_{\re{\phi(\fraku)} c}} = \bar{\stars{q_1}{\re{\phi(\fraku)}}}< w\, \text{ and }\,
\bar{q|_{a \re{\phi(\frakv)} }} = \bar{ \stars{q_2}{\re{\phi(\frakv)}}} < w.$$
This implies that
$$ q|_{\re{\phi(\fraku)} c},  q|_{a \re{\phi(\frakv)} }\in \bfk Y\, \text{ and }\, \re{\phi(\fraku)} c, a \re{\phi(\frakv)}\in \bfk Y.$$
Together with $\re{\phi(\fraku)} c \downarrow_\phi a \re{\phi(\frakv)}$ and Theorem~\mref{thm:sgoodtype}.(\mref{it:rb3}), we have
$$\re{\phi(\fraku)} c \downarrow_{\Pi_{\bfk Y}} a \re{\phi(\frakv)}\, \text{ and }\, \re{\phi(\fraku)} c - a \re{\phi(\frakv)} \astarrow_{\Pi_{\bfk Y}} 0,$$
where the last step employs the fact that $\Pi_{\bfk Y} \leq \Pi_\phi$ is confluent and Remark~\mref{re:tozero}.
Thus
$$q|_{\re{\phi(\fraku)} c} - q|_{a \re{\phi(\frakv)} } = q|_{\re{\phi(\fraku)} c - a \re{\phi(\frakv)} } \astarrow_{\Pi_{\bfk Y}} 0,
\,\text{ that is, }\, \stars{q_1}{\re{\phi(\fraku)}} -  \stars{q_2}{\re{\phi(\frakv)}} \astarrow_{\Pi_{\bfk Y}} 0.$$
By $\Pi_{\bfk Y}\leq \Pi_\phi$ and Lemma~\mref{lem:cons},
$$\stars{q_1}{\re{\phi(\fraku)}} -  \stars{q_2}{\re{\phi(\frakv)}} \astarrow_{\phi} 0\,\text{ and }\, d\stars{q_1}{\re{\phi(\fraku)}} -  d\stars{q_2}{\re{\phi(\frakv)}} \astarrow_{\phi} 0,$$
contradicting Eq.~(\mref{eq:contra}).

\smallskip

\noindent {\bf Case III: Nested placements}.
By symmetry, we may suppose that there is $q \in \frakM^\star(Z)$ such that $\stars{q_1}q = q_2$.
Let us first consider $q = \star$.
Then $q_1 = q_2$. Since $ \stars{q_1}{\bar{\phi(\fraku)}} = \stars{q_2}{\bar{\phi(\frakv)}}$, by Lemma~\mref{lem:mstars}, we get $\bar{\phi(\fraku)} = \bar{\phi(\frakv)}$.
Taking $a=c=1$ and $b= \lbar{\phi(\fraku)}=\lbar{\phi(\frakv)}$ in Theorem~\mref{thm:sgoodtype}~(\mref{it:rb3}), we have
$$\re{\phi(\fraku)} \downarrow_\phi \re{\phi(\frakv)},\, \stars{q_1}{\re{\phi(\fraku)}} \downarrow_\phi \stars{q_2}{\re{\phi(\frakv)}} \,\text{ and }\, \stars{q_1}{\re{\phi(\fraku)}} \downarrow_{\Pi_{\bfk Y}} \stars{q_2}{\re{\phi(\frakv)}},$$
where the second confluence follows from $q_1 = q_2$ and the last confluence from Eq.~(\mref{eq:less1}).
Since $\Pi_{\bfk Y}$ is confluent, it follows from Remark~\mref{re:tozero} and Lemma~\mref{lem:cons} that
$$\stars{q_1}{\re{\phi(\fraku)}}- \stars{q_2}{\re{\phi(\frakv)}} \astarrow_{\Pi_{\bfk Y}} 0\,
\text{ and }\, d\stars{q_1}{\re{\phi(\fraku)}}- d\stars{q_2}{\re{\phi(\frakv)}} \astarrow_{\Pi_{\bfk Y}} 0,$$
contradicting Eq.~(\mref{eq:contra}).

Consider next $q \ne \star$. So $q_1\neq q_2$ and $\bar{\phi(\fraku)} \ne \bar{\phi(\frakv)}$. From
$$\stars{q_1}{\bar{\phi(\fraku)}}
= \stars{q_2}{\bar{\phi(\frakv)}}=  \stars{ q_1}{\stars{q}{\bar{\phi(\frakv)}}},$$
we have $\bar{\phi(\fraku)} = q|_{\bar{\phi(\frakv)}}$ by Lemma~\mref{lem:mstars}. Using Theorem~\mref{thm:sgoodtype}(\mref{it:incl}),
there are some $u_i$ with $1\leq i\leq k$ and $q'\in \frakM^\star(Z)$ such that $u_i = \stars{q'}{\bar{\phi(\frakv)}}$. Write
\begin{equation}
\bar{\phi(\fraku)} = \stars{p}{u_1, \ldots, u_k}\, \text{ for some } p\in \frakM^\star(Z).
\mlabel{eq:pbar}
\end{equation}
Then
\begin{align*}
\stars{q}{\bar{\phi(\frakv)}} = \bar{\phi(\fraku)} =  \stars{p}{u_1, \ldots, u_k}
= \stars{p}{u_1, \ldots, u_{i-1},\; \stars{q'}{\bar{\phi(\frakv)}},\; u_{i+1}, \ldots,  u_k}
= \stars{(\stars{p}{u_1, \ldots, u_{i-1},q', u_{i+1}, \ldots,  u_k})}{\bar{\phi(\frakv)}}.
\end{align*}
This implies that
\begin{equation}
q = \stars{p}{u_1, \ldots, u_{i-1},\; q',\; u_{i+1}, \ldots,  u_k} = \bar{\phi(u_1, \ldots, u_{i-1}, q', u_{i+1}, \ldots, u_k)},
\mlabel{eq:qq}
\end{equation}
where the second step employs Eq.~(\mref{eq:pbar}).
From Eq.~(\mref{eq:sumuv}), we may write
\begin{equation}
\stars{q_1}{\re{\phi(\fraku)}} = \sum_{t=1}^m r_t \stars{q_1}{\stars{p_t}{\fraku}},
\mlabel{eq:qu}
\end{equation}
where \begin{align}
\fraku = (u_1, \ldots, u_k) = (u_1, \cdots, u_{i-1}, \stars{q'}{\bar{\phi(\frakv)}}, u_{i+1}, \ldots, u_k).
\mlabel{eq:equ}
\end{align}
Write
\begin{equation}
\fraku' = (u_1, \ldots, u_{i-1}, q'|_{\re{\phi(\frakv)}}, u_{i+1}, \ldots, u_k)\,\text{ and }\, \fraku'_{s}= (u_1, \ldots, u_{i-1},\stars{q'}{\stars{p_s}{\frakv}}, u_{i+1}, \ldots, u_k)
\mlabel{eq:uprime}
\end{equation}
for $1\leq s\leq m$. Then
$$\stars{q_1}{\stars{p_t}{\fraku}} \to_{\Pi_{\bfk Y}} \stars{q_1}{\stars{p_t}{\fraku'}} = \sum_{s=1}^m r_s \stars{q_1}{\stars{p_t}{\fraku'_{s} }} \, \text{ for } \, 1\leq t\leq m,$$
where the first rewriting step follows from Eqs.~(\mref{eq:uvless}) and~(\mref{eq:equ}), and the equation from Eq.~(\mref{eq:sumuv}).
This implies that
\begin{equation*}
\stars{q_1}{\stars{p_t}{\fraku}} \downarrow_{\Pi_{\bfk Y}} \sum_{s=1}^m r_s \stars{q_1}{\stars{p_t}{\fraku'_{s} }} \, \text{ for }\, 1\leq t\leq m.
\end{equation*}
By Lemmas~\mref{lem:cons} and~\mref{lem:red}(\mref{it:trans}),
\begin{equation}
\sum_{t=1}^m  r_t \stars{q_1}{\stars{p_t}{\fraku}} \downarrow_{\Pi_{\bfk Y}} \sum_{t, s=1}^m r_tr_s\, \stars{q_1}{\stars{p_t}{\fraku'_{s} }}.
\mlabel{eq:nest1}
\end{equation}
On the other hand,
\begin{equation}
\begin{aligned}
& \stars{q_2}{\re{\phi(\frakv)}}  = \stars{q_1}{q|_{\re{\phi(\frakv)}}} = \sum_{s=1}^m r_s
\stars{q_1}{q|_{p_s|_v}}
=\sum_{s=1}^m r_s \stars{q_1}{\bar{\phi(u_1, \cdots, u_{i-1},\, \stars{q'}{\stars{p_s}{\frakv}},\, u_{i+1}, \cdots, u_k) }} = \sum_{s=1}^m r_s \stars{q_1}{\bar{\phi(\fraku'_s) }},
\end{aligned}
\mlabel{eq:qv}
\end{equation}
where the first equation follows from $q_2= \stars{q_1}{q}$, the second from Eq.~(\mref{eq:sumuv}), the third from Eq.~(\mref{eq:qq}) and the fourth from Eq.~(\mref{eq:uprime}).
Since $\leq$ is a monomial order and $\stars{p_s}{\frakv} < \bar{\phi(v)}$, we have
$$\stars{q'}{\stars{p_s}{\frakv}} < \stars{q'}{\bar{\phi(v)}} \,\text{ and }\, \bar{\phi(\fraku'_s)} < \bar{\phi(\fraku)} \, \text{ for } \, 1\leq s\leq m,$$
where the second inequality employs Eqs~(\mref{eq:equ}) and~(\mref{eq:uprime}). This implies that
\begin{align*}
\stars{q_1}{\bar{\phi(\fraku'_s)}} < \stars{q_1}{\bar{\phi(\fraku)}} = w\, \text{ and}\, \stars{q_1}{\bar{\phi(\fraku'_s)}}\in \bfk Y \, \text{ for } \, 1\leq s\leq m.
\end{align*}
So
$$ \stars{q_1}{\bar{\phi(\fraku'_s) }}
\to_{\Pi_{\bfk Y}} \stars{q_1}{\re{\phi(\fraku'_s)}} = \sum_{t=1}^m r_t \stars{q_1}{\stars{p_t}{\fraku'_s}} \, \text{ for } \, 1\leq s\leq m$$
and
\begin{equation*}
\stars{q_1}{\bar{\phi(\fraku'_s) }} \downarrow_{\Pi_{\bfk Y}} \sum_{t=1}^m r_t \stars{q_1}{\stars{p_t}{\fraku'_s}}\, \text{ for }\, 1\leq s \leq m.
\end{equation*}
Again applying Lemmas~\mref{lem:cons} and~\mref{lem:red}(\mref{it:trans}), we have
\begin{equation}
\sum_{s=1}^m r_s \stars{q_1}{\bar{\phi(\fraku'_s) }} \downarrow_{\Pi_{\bfk Y}} \sum_{t,s =1}^m r_t r_s \stars{q_1}{\stars{p_t}{\fraku'_s}}.
\mlabel{eq:nest2}
\end{equation}
Since $\Pi_{\bfk Y}$ is confluent, by Lemma~\mref{lem:red}~(\mref{it:trans}), Eqs.~(\mref{eq:nest1}) and~(\mref{eq:nest2}) we obtain
$$\sum_{t=1}^m  r_t \stars{q_1}{\stars{p_t}{\fraku}}  \downarrow_{\Pi_{\bfk Y}} \sum_{s=1}^m r_s \stars{q_1}{\bar{\phi(\fraku'_s) }}.$$
Then Remark~\mref{re:tozero} yields,
$$ \sum_{t=1}^m r_t \stars{q_1}{\stars{p_t}{\fraku}} -\sum_{s=1}^m  r_s \stars{q_1}{\bar{\phi(\fraku'_s) }} \astarrow_{\Pi_{\bfk Y}} 0 .$$
By Eqs.~(\mref{eq:qu}) and~(\mref{eq:qv}), this is equivalent to
$$\stars{q_1}{\re{\phi(\fraku)}}  - \stars{q_2}{\re{\phi(\frakv)}}  \astarrow_{\Pi_{\bfk Y}} 0.$$
Hence from Lemma~\mref{lem:cons} and $\Pi_{\bfk Y}\leq \Pi_\phi$, we conclude
$$d\stars{q_1}{\re{\phi(\fraku)}}  - d\stars{q_2}{\re{\phi(\frakv)}}  \astarrow_{\Pi_{\bfk Y}} 0\, \text{ and }\, d\stars{q_1}{\re{\phi(\fraku)}}  - d\stars{q_2}{\re{\phi(\frakv)}}  \astarrow_{\phi} 0,$$
contradicting Eq.~(\mref{eq:contra}).

Thus $\frakC \neq \emptyset$ leads to contradiction in all possible cases. This completes the proof of Theorem~\ref{thm:sgoodtype}.
\end{proof}

\smallskip

\noindent {\bf Acknowledgements}: This work was supported by the and the National Science Foundation of US (Grant No. DMS~1001855) and National Natural Science Foundation of China (Grant No. 11201201, 11371177 and 11371178).

\end{document}